\documentclass[12pt]{amsart}

\usepackage{enumerate}
\usepackage{amssymb,amsfonts,amsmath,amsthm,mathrsfs}
\usepackage{dsfont}
\usepackage{graphicx}
\usepackage{xcolor}

\usepackage[colorlinks=true,
linkcolor=blue,
citecolor=blue,
urlcolor=blue]{hyperref}

\usepackage{setspace}
\usepackage{geometry}
\numberwithin{equation}{section}

\theoremstyle{plain}
\newtheorem{theorem}{Theorem}[section]
\newtheorem{proposition}[theorem]{Proposition}
\newtheorem{corollary}[theorem]{Corollary}

\theoremstyle{definition}
\newtheorem{definition}[theorem]{Definition}
\newtheorem{example}[theorem]{Example}
\newtheorem{remark}[theorem]{Remark}

\DeclareMathOperator{\Rep}{Rep}

\renewcommand{\Re}{\operatorname{Re}}

\newcommand{\N}{{\mathbb{N}}}

\newcommand{\R}{\mathbb{R}}
\newcommand{\C}{\mathbb{C}}

\newcommand{\DD}{\mathscr{D}'}

\def\jp#1{{\left\langle{#1}\right\rangle}}

\title[Vekua-Type Operators on Compact Lie Groups]
{Vekua-Type Operators on Compact Lie Groups: \\ Hypoellipticity, Solvability, and Self-Duality}

\author[A. Kirilov]{Alexandre Kirilov}
\address{
	Departamento de Matem\'atica, Universidade Federal do Paran\'a,
	Caixa Postal 19081, Curitiba, 81531-990, PR, Brazil
}
\email{akirilov@ufpr.br}

\author[R. Paleari da Silva]{Ricardo Paleari da Silva}
\address{
	Colegiado de Matem\'atica, Unespar, Rua Comendador Correa Junior, 117,
	Paranagu\'a, 83203-560, PR, Brazil
}
\email{ricardo.paleari@unespar.edu.br}



\begin{document}

	\begin{abstract}
		We study global hypoellipticity and global solvability for Vekua-type operators associated with diagonal left-invariant operators on compact Lie groups. The conjugation term produces, on the Fourier side, a family of coupled systems whose determinants govern both regularity and solvability. Under a natural non-self-duality assumption, we obtain complete Diophantine-type characterizations of global hypoellipticity and global solvability, the latter on the natural space of admissible data. We then show how the theory must be modified in the self-dual setting by carrying out a complete analysis of the model case \(\mathbb S^3\simeq SU(2)\), where conjugation acts inside each representation block. The resulting criteria exhibit two Fourier-side mechanisms for Vekua-type operators on compact Lie groups: coupling between distinct conjugate representations and coupling inside self-dual representation blocks. We also present examples on product groups illustrating how these mechanisms may coexist and how global solvability may hold even when global hypoellipticity fails.
	\end{abstract}
	
	\maketitle

	
	\maketitle 
	
	{
		\small 
		\tableofcontents 
	} 
	
	\section{Introduction}
	
	The purpose of this paper is to study global hypoellipticity and global
	solvability for Vekua-type operators on compact Lie groups. More precisely, we
	consider operators of the form
	\[
	Pu=Lu-q\,u-p\,\overline u,
	\]
	where \(L\) is a left-invariant operator on a compact Lie group \(G\),
	the constants \(p,q\in\C\), and \(p\neq0\). The presence of the conjugation term makes \(P\) a
	\(\R\)-linear operator, but in general not \(\C\)-linear. This simple change has a
	substantial effect on the Fourier analysis of the equation: the Fourier
	coefficient of \(u\) at a representation is no longer governed independently,
	but is coupled with the coefficient coming from the conjugate representation, or
	with a symmetric entry inside the same representation block when the
	representation is self-dual.
	
	The central problem is to determine when the equation
	\[
	Pu=f
	\]
	has a distributional solution, and when the regularity of \(Pu\) guarantees the
	regularity of \(u\). In the complex-linear setting, such questions are often
	reduced to lower bounds for scalar symbols or matrix symbols. For Vekua-type
	operators, the conjugation term produces \(2\times2\) systems on the Fourier
	side. The determinant of these systems is the main object of the paper: its
	zeros give the algebraic compatibility conditions for solvability, while its
	small nonzero values are the small divisors that control global regularity.
	
	We first develop this analysis for Vekua-type perturbations of diagonal
	left-invariant operators on compact Lie groups. The class of diagonal operators
	considered here includes many natural left-invariant differential operators and
	allows arbitrary order. Under a non-self-duality assumption on the relevant
	representation classes, the Fourier-side equation decouples into scalar
	\(2\times2\) systems involving a representation and its conjugate. We prove that
	a Diophantine-type lower bound for the associated determinants characterizes
	global hypoellipticity. We also prove the corresponding characterization of
	global solvability, where solvability is understood on the natural space of
	admissible data determined by the algebraic compatibility conditions at the
	zeros of the determinant.
	
	A second main point of the paper is that the self-dual case cannot be treated as
	a formal variant of the non-self-dual one. If a representation is equivalent to
	its conjugate, the conjugation term does not couple two distinct representation
	classes. Instead, it acts inside the same representation space. The basic model
	for this phenomenon is
	\[
	\mathbb S^3\simeq SU(2),
	\]
	where all irreducible representations are self-dual. In the standard Fourier
	basis of \(SU(2)\), conjugation relates the entries
	\[
	(m,n)
	\quad\text{and}\quad
	(-m,-n)
	\]
	inside each representation block. We derive the corresponding internal
	\(2\times2\) systems and obtain complete characterizations of global
	hypoellipticity and global solvability on \(\mathbb S^3\). Thus the role played
	by the conjugate representation in the non-self-dual case is replaced, in the
	self-dual case, by an explicit blockwise symmetry.
	
	This distinction is not merely technical. It provides a way to analyze
	Vekua-type operators on compact Lie groups by identifying the Fourier-side
	mechanism produced by conjugation in each representation class. 
	
	On product
	groups such as \(\mathbb S^3\times\mathbb T^1,\) 
	both mechanisms may occur simultaneously: nonzero toroidal frequencies couple
	distinct conjugate representation classes, whereas the zero toroidal frequency
	leads to an internal coupling inside the \(\mathbb S^3\)-blocks. The examples
	at the end of the paper illustrate this interaction. They also show that a
	Vekua-type perturbation may restore global hypoellipticity for an operator whose
	diagonal part is not globally hypoelliptic, and that global solvability may hold
	even when global hypoellipticity fails.
	
	The study of global hypoellipticity and global solvability has its roots in the
	classical works of Greenfield and Wallach on vector fields on tori, where
	Diophantine conditions were shown to govern global regularity
	\cite{GW1972_pams,greenfield1973globally,GW1973_tams}. Related questions on
	solvability, global regularity, and involutive systems were developed, for
	instance, in
	\cite{bergamasco2014solvability,CH1977_pams,DM2016_mana,de2021solvability}.
	Fourier methods on compact Lie groups, based on the Peter--Weyl decomposition
	and on the global symbolic calculus of Ruzhansky and Turunen
	\cite{RT2010_book}, have provided a framework for studying global properties of
	left-invariant operators, vector fields, evolution operators, and diagonal
	systems; see, among others,
	\cite{da2025diagonal,KMR2020_bsm,KMR2021_jfa,KKM2024,KMP2021_jde}. Motivated by
	the theory of generalized analytic functions \cite{vekua2014generalized},
	Vekua-type equations have also been considered in global settings, including
	periodic operators and compact Lie groups
	\cite{de2022regularity,kirilov2025denjoy,kirilov2026solvability}.
	
	The present work continues this line in two directions. First, it treats
	Vekua-type perturbations of arbitrary-order diagonal operators on compact Lie
	groups and gives determinant-based criteria for global hypoellipticity and
	global solvability in the non-self-dual regime. Second, it develops a separate
	blockwise analysis for the self-dual structure of
	\(\mathbb S^3\simeq SU(2)\). Together, these two analyses provide a framework
	for handling the conjugation term according to how conjugation acts on the
	Fourier side: either between distinct conjugate representations or inside a
	self-dual representation block.
	
	The paper is organized as follows. Section~\ref{overview} recalls the notation
	and basic facts from Fourier analysis on compact Lie groups. In
	Section~\ref{Section_Vekua_operators}, we introduce Vekua-type operators
	associated with diagonal operators and study their global hypoellipticity and
	global solvability in the non-self-dual setting. Section~\ref{Section_S3} is
	devoted to the self-dual structure of Vekua-type operators on
	\(\mathbb S^3\simeq SU(2)\), where we obtain the corresponding blockwise
	characterizations. Finally, Section~\ref{Section_examples} presents examples
	illustrating the criteria, including product groups where the non-self-dual and
	self-dual mechanisms coexist.

	\section{Fourier Analysis on Compact Lie Groups}\label{overview}
	
	In this section, we fix the notation and recall the basic facts on Fourier
	analysis on compact Lie groups that will be used throughout the paper. For a
	detailed account, including proofs, we refer the reader to \cite{RT2010_book}.
	
	Let \(G\) be a compact Lie group endowed with its normalized Haar measure
	\(\mu\). We denote by \(\Rep(G)\) the set of continuous irreducible unitary
	representations of \(G\), and by
	\[
	\widehat G:=\Rep(G)/{\sim}
	\]
	its unitary dual, that is, the set of equivalence classes of irreducible
	unitary representations.
	
	For each class \(\Xi\in\widehat G\), we fix a representative
	\[
	\xi:G\to U(d_\xi),
	\]
	where \(d_\xi\) denotes the dimension of \(\xi\). Writing
	\[
	\xi=(\xi_{mn})_{1\leq m,n\leq d_\xi},
	\]
	the Peter--Weyl theorem asserts that
	\[
	\bigcup_{\Xi\in\widehat G}
	\left\{\sqrt{d_\xi}\,\xi_{mn}:1\leq m,n\leq d_\xi\right\}
	\]
	is an orthonormal basis of \(L^2(G)\).
	
	Throughout the paper, one representative \(\xi\) is fixed for each equivalence
	class \(\Xi\in\widehat G\). In particular situations, additional assumptions on
	these representatives may be imposed.
	
	For \(f\in L^1(G)\) and \(\Xi=[\xi]\in\widehat G\), the Fourier coefficient of
	\(f\) at \(\xi\) is the matrix
	\[
	\widehat f(\xi)
	:=\int_G f(g)\xi(g)^*\,d\mu(g).
	\]
	This definition depends on the choice of representative only up to unitary
	conjugation; all statements below are invariant under this change, or are to be
	understood after fixing the above system of representatives.
	
	We denote by \(C^\infty(G)\) the space of smooth functions on \(G\), endowed
	with its usual Fréchet topology, and by \(\DD(G)\) the space of distributions
	on \(G\), that is, the topological dual of \(C^\infty(G)\).
	
	Let \(\Delta\) be the Laplace--Beltrami operator on \(G\). For each
	\(\Xi=[\xi]\in\widehat G\), all matrix coefficients \(\xi_{mn}\) are
	eigenfunctions of \(\Delta\) with the same eigenvalue \(\lambda(\xi)\leq 0\).
	We write
	\[
	\jp{\xi}:=(1-\lambda(\xi))^{1/2},
	\]
	which is the eigenvalue of the positive operator \((I-\Delta)^{1/2}\)
	associated with the representation class \(\Xi\). In particular,
	\(\jp{\xi}\geq 1\).
	
	For \(u\in\DD(G)\) and \(\Xi=[\xi]\in\widehat G\), the Fourier coefficient of
	\(u\) at \(\xi\) is defined by
	\[
	\widehat u(\xi):=\langle u,\xi^*\rangle.
	\]
	Again, this is well defined up to unitary conjugation.
	
	The Peter--Weyl decomposition allows one to characterize distributions and
	smooth functions in terms of the growth or decay of their Fourier coefficients.
	More precisely, suppose that to each \(\Xi=[\xi]\in\widehat G\) we associate a
	matrix
	\[
	x(\xi)=(x(\xi)_{mn})\in\C^{d_\xi\times d_\xi}
	\]
	and that there exist constants \(M>0\) and \(N>0\) such that
	\begin{equation}\label{dist}
		|x(\xi)_{mn}|\leq M\jp{\xi}^{N},
	\end{equation}
	for all \(1\leq m,n\leq d_\xi\) and all \(\Xi=[\xi]\in\widehat G\). Then the
	series
	\[
	u:=\sum_{\Xi=[\xi]\in\widehat G}
	d_\xi\sum_{m,n}x(\xi)_{mn}\xi_{nm}
	\]
	converges in \(\DD(G)\) and defines a distribution satisfying
	\[
	\widehat u(\xi)_{mn}=x(\xi)_{mn}.
	\]
	
	Conversely, if \(u\in\DD(G)\), then there exist constants \(M>0\) and \(N>0\)
	such that
	\[
	|\widehat u(\xi)_{mn}|\leq M\jp{\xi}^{N},
	\]
	for all \(1\leq m,n\leq d_\xi\) and all \(\Xi=[\xi]\in\widehat G\).
	
	Similarly, a distribution \(u\in\DD(G)\) is smooth if and only if its Fourier
	coefficients are rapidly decreasing: for every \(N>0\), there exists
	\(M_N>0\) such that
	\begin{equation}\label{smooth1}
		|\widehat u(\xi)_{mn}|\leq M_N\jp{\xi}^{-N},
	\end{equation}
	for all \(1\leq m,n\leq d_\xi\) and all \(\Xi=[\xi]\in\widehat G\). In this
	case, the associated Fourier series converges in the usual sense for smooth
	functions, and in particular in \(L^2(G)\), where the Plancherel formula holds.
	
	Let \(P:C^\infty(G)\to C^\infty(G)\) be a continuous linear operator. Its
	symbol is defined by
	\[
	\sigma_P(g,\xi):=\xi(g)^*(P\xi)(g),
	\]
	where \(P\xi\) denotes the matrix whose entries are
	\[
	(P\xi)_{mn}:=P(\xi_{mn}),\qquad 1\leq m,n\leq d_\xi.
	\]
	
	If \(P\) is left-invariant and extends continuously to \(\DD(G)\), then
	\(\sigma_P\) is independent of \(g\). In this case, for every \(u\in\DD(G)\)
	and every \(\Xi=[\xi]\in\widehat G\), one has
	\begin{equation}\label{multiply}
		\widehat{Pu}(\xi)=\sigma_P(\xi)\widehat u(\xi).
	\end{equation}
	
	Finally, for \(u\in\DD(G)\), we define the conjugate distribution \(\overline u\)
	by
	\[
	\langle\overline u,\varphi\rangle
	:=\overline{\langle u,\overline\varphi\rangle},
	\qquad \varphi\in C^\infty(G).
	\]
	Then \(\overline u\in\DD(G)\). If \(\xi\in\Rep(G)\), the map
	\[
	\overline\xi(g):=\overline{\xi(g)},\qquad g\in G,
	\]
	defines the conjugate representation of \(\xi\), with
	\(d_{\overline\xi}=d_\xi\).
	
	Whenever the conjugate representation \(\overline\xi\) appears in a Fourier
	coefficient, we shall either regard \(\widehat u(\overline\xi)\) as the Fourier
	coefficient computed with respect to the representation \(\overline\xi\) itself,
	or choose the representatives in conjugate classes accordingly. Since different
	choices of equivalent representatives are related by unitary conjugation, this
	convention does not affect the estimates used below.
	
	Since \((\xi^*)_{k\ell}=\overline{\xi_{\ell k}}\), the Fourier coefficients of
	\(u\) and \(\overline u\) are related by
	\[
	\widehat{\overline u}(\xi)_{k\ell}
	=\langle\overline u,\overline{\xi}_{\ell k}\rangle
	=\overline{\langle u,\xi_{\ell k}\rangle}
	=\overline{\langle u,((\overline\xi)^*)_{k\ell}\rangle}
	=\overline{\widehat u(\overline\xi)_{k\ell}},
	\]
	and therefore
	\begin{equation}\label{conj}
		\widehat{\overline u}(\xi)
		=
		\overline{\widehat u(\overline\xi)}, \qquad [\xi]\in\widehat G.
	\end{equation}

	\section{Vekua-Type Operators on Compact Lie Groups}
	\label{Section_Vekua_operators}
	
	Let \(G\) be a compact Lie group, and let
	\[
	L:\DD(G)\to\DD(G)
	\]
	be a continuous left-invariant operator such that
	\[
	L(C^\infty(G))\subset C^\infty(G).
	\]
	
	Given \(p,q\in\C\) with \(p\neq0\), we define the Vekua-type operator associated
	with \(L\) by
	\begin{equation}\label{vekua_op}
		Pu:=Lu-q\,u-p\,\overline u,
		\qquad u\in\DD(G).
	\end{equation}
	
	Because of the term involving \(\overline u\), the operator \(P\) is
	\(\R\)-linear, but in general it is not \(\C\)-linear.
	
	Our purpose in this section is to study the global properties of \(P\), namely
	its global hypoellipticity and global solvability. The main difficulty comes
	from the conjugation term, which couples the Fourier coefficients of \(u\) with
	those of \(\overline u\). We begin by deriving the Fourier-side system
	associated with the equation \(Pu=f\).
	
	Let \(u,f\in\DD(G)\) and suppose that
	\[
	Pu=f.
	\]
	
	Taking Fourier coefficients and using \eqref{multiply}, we obtain, for every
	\([\xi]\in\widehat G\),
	\[
	\widehat{Pu}(\xi)
	=
	\widehat{Lu}(\xi)
	-q\,\widehat u(\xi)
	-p\,\widehat{\overline u}(\xi).
	\]
	
	Thus,
	\[
	\big(\sigma_L(\xi)-qI_{d_\xi}\big)\widehat u(\xi)
	-p\,\widehat{\overline u}(\xi)
	=
	\widehat f(\xi).
	\]
	
	Applying the same argument to the conjugate representation \(\overline\xi\) and
	using \eqref{conj}, we have
	\[
	\widehat{Pu}(\overline\xi)
	=
	\big(\sigma_L(\overline\xi)-qI_{d_\xi}\big)
	\widehat u(\overline\xi)
	-p\,\widehat{\overline u}(\overline\xi)
	=
	\widehat f(\overline\xi).
	\]
	
	Equivalently,
	\[
	\big(\sigma_L(\overline\xi)-qI_{d_\xi}\big)
	\overline{\widehat{\overline u}(\xi)}
	-p\,\overline{\widehat u(\xi)}
	=
	\overline{\widehat{\overline f}(\xi)}.
	\]
	
	Taking complex conjugates yields
	\[
	-\overline p\,\widehat u(\xi)
	+
	\big(\overline{\sigma_L(\overline\xi)}
	-\overline q I_{d_\xi}\big)
	\widehat{\overline u}(\xi)
	=
	\widehat{\overline f}(\xi).
	\]
	
	Therefore, for each \([\xi]\in\widehat G\), the Fourier coefficients of \(u\)
	and \(\overline u\) satisfy the system
	\begin{equation}\label{mainsystem}
		\left\{
		\begin{array}{l}
			\big(\sigma_L(\xi)-qI_{d_\xi}\big)\widehat u(\xi)
			-p\,\widehat{\overline u}(\xi)
			=
			\widehat f(\xi),
			\\[0.4em]
			-\overline p\,\widehat u(\xi)
			+
			\big(\overline{\sigma_L(\overline\xi)}
			-\overline q I_{d_\xi}\big)
			\widehat{\overline u}(\xi)
			=
			\widehat{\overline f}(\xi).
		\end{array}
		\right.
	\end{equation}
	
	In order to solve this system explicitly, we restrict our attention to a class
	of left-invariant operators whose symbols can be diagonalized in a way
	compatible with conjugation.
	
	\begin{definition}\label{diagonal}
		Let \(L:\DD(G)\to\DD(G)\) be a continuous left-invariant operator preserving
		smooth functions. We say that \(L\) is \emph{diagonal} if the following
		conditions hold:
		\begin{enumerate}
			\item[1.] For each \(\Xi\in\widehat G\), there exists a representative
			\(\xi\in\Xi\) such that \(\sigma_L(\xi)\) is diagonal. Moreover, whenever
			the conjugate representation \(\overline\xi\) is used, the symbols are
			compatible with conjugation in the sense that
			\[
			\sigma_L(\overline\xi)
			=
			\overline{\sigma_L(\xi)}.
			\]
			We write
			\(
			\sigma_L(\xi)
			=
			\operatorname{diag}
			\big(\sigma_1(\xi),\ldots,\sigma_{d_\xi}(\xi)\big).
			\)
			
			\item[2.] There exist constants \(C_\sigma>0\) and \(K\in\mathbb N_0\) such
			that
			\begin{equation}\label{condsimbolo}
				|\sigma_k(\xi)|
				\leq
				C_\sigma\jp{\xi}^{K},
			\end{equation}
			for every \([\xi]\in\widehat G\) and every \(1\leq k\leq d_\xi\).
		\end{enumerate}
	\end{definition}
	
	\begin{example}
		Let \(\{X_1,\ldots,X_n\}\) be a finite family of real left-invariant vector
		fields on a compact Lie group \(G\). For a multi-index
		\(\alpha=(\alpha_1,\ldots,\alpha_n)\in\mathbb N_0^n\), set
		\[
		X^\alpha:=X_1^{\alpha_1}\circ\cdots\circ X_n^{\alpha_n}.
		\]
		Given \(r\in\mathbb N\) and coefficients \(a_\alpha\in\R\), the differential
		operator
		\[
		L:=\sum_{1\leq |\alpha|\leq r}a_\alpha X^\alpha
		\]
		defines a continuous left-invariant operator on \(\DD(G)\) and preserves
		\(C^\infty(G)\).
		
		Since the vector fields \(X_j\) are real, their symbols satisfy
		\[
		\sigma_{X_j}(\overline\xi)
		=
		\overline{\sigma_{X_j}(\xi)},
		\qquad j=1,\ldots,n.
		\]
		Hence the same compatibility with conjugation holds for every such operator
		\(L\):
		\[
		\sigma_L(\overline\xi)
		=
		\overline{\sigma_L(\xi)}.
		\]
		
		Moreover, each \(X_j\) is skew-symmetric on \(L^2(G)\), and therefore
		\(iX_j\) is symmetric. Consequently, for a single real left-invariant vector
		field \(X\), one may choose, in each representation class
		\(\Xi\in\widehat G\), a unitary representative for which
		\[
		\sigma_X(\xi)
		=
		\operatorname{diag}
		\big(i\mu_1(\xi),\ldots,i\mu_{d_\xi}(\xi)\big),
		\qquad
		\mu_\ell(\xi)\in\R.
		\]
		
		It follows that every polynomial in \(X\) with real coefficients is diagonal
		in the sense of Definition~\ref{diagonal}.
		
		More generally, if the vector fields \(X_1,\ldots,X_n\) commute pairwise,
		then, for each \(\Xi\in\widehat G\), the matrices
		\(\sigma_{X_1}(\xi),\ldots,\sigma_{X_n}(\xi)\) are commuting normal matrices.
		Hence they can be simultaneously diagonalized by a unitary change of basis,
		and every operator of the above form is diagonal.
	\end{example}
	
	From now on, assume that \(L\) is diagonal, and let \(P\) be the Vekua-type
	operator defined in \eqref{vekua_op}.
	
	For each \(\Xi\in\widehat G\), we fix a representative \(\xi\in\Xi\) satisfying
	Definition~\ref{diagonal}. Since
	\[
	\sigma_L(\overline\xi)
	=
	\overline{\sigma_L(\xi)},
	\]
	we have
	\[
	\overline{\sigma_L(\overline\xi)}
	=
	\sigma_L(\xi).
	\]
	
	Thus the matrix system \eqref{mainsystem} decouples into scalar systems for the
	entries of the Fourier coefficients. For every \([\xi]\in\widehat G\) and every
	\(1\leq k,\ell\leq d_\xi\), we obtain
	\begin{equation}\label{subsystem}
		\left\{
		\begin{array}{l}
			(\sigma_k(\xi)-q)\widehat u(\xi)_{k\ell}
			-p\,\widehat{\overline u}(\xi)_{k\ell}
			=
			\widehat f(\xi)_{k\ell},
			\\[0.4em]
			-\overline p\,\widehat u(\xi)_{k\ell}
			+
			(\sigma_k(\xi)-\overline q)
			\widehat{\overline u}(\xi)_{k\ell}
			=
			\widehat{\overline f}(\xi)_{k\ell}.
		\end{array}
		\right.
	\end{equation}
	
	The determinant of this \(2\times2\) system is
	\begin{equation}\label{discriminant}
		\Delta_k(\xi)
		:=
		(\sigma_k(\xi)-q)(\sigma_k(\xi)-\overline q)
		-|p|^2.
	\end{equation}
	
	Equivalently,
	\[
	\Delta_k(\xi)
	=
	\sigma_k(\xi)^2
	-2\Re(q)\sigma_k(\xi)
	+|q|^2-|p|^2.
	\]
	
	Eliminating \(\widehat{\overline u}(\xi)_{k\ell}\) from
	\eqref{subsystem}, or equivalently applying Cramer's rule whenever
	\(\Delta_k(\xi)\neq0\), gives the identity
	\begin{equation}\label{cramer}
		\Delta_k(\xi)\widehat u(\xi)_{k\ell}
		=
		(\sigma_k(\xi)-\overline q)\widehat f(\xi)_{k\ell}
		+
		p\,\widehat{\overline f}(\xi)_{k\ell}.
	\end{equation}
	This identity is valid for all \([\xi]\in\widehat G\) and all
	\(1\leq k,\ell\leq d_\xi\), regardless of whether \(\Delta_k(\xi)\) vanishes.
	
	Thus, the determinant \(\Delta_k(\xi)\) is the small divisor associated with the
	Fourier-side Vekua system. In order to recover regularity of \(u\) from the
	regularity of \(Pu=f\), the rapid decay of the right-hand side of
	\eqref{cramer} must be transferred to the coefficient
	\(\widehat u(\xi)_{k\ell}\). This transfer is possible only when division by
	\(\Delta_k(\xi)\) does not destroy the decay of the Fourier coefficients.
	Hence the obstruction to global regularity is not merely the vanishing of
	\(\Delta_k(\xi)\), but also the possibility that \(\Delta_k(\xi)\) becomes too
	small along high-frequency representations. This leads naturally to the
	following Diophantine-type lower bound.
	
	\begin{definition}\label{DC}
		Let \(L\) be a diagonal operator in \(\DD(G)\) in the sense of
		Definition~\ref{diagonal}, let \(p,q\in\C\) with \(p\neq0\), and let \(P\)
		be the associated Vekua-type operator defined by \eqref{vekua_op}. We say
		that \(P\) satisfies condition \emph{(DC)} if there exist constants
		\(C>0\), \(M>0\), and \(R>0\) such that
		\[
		|\Delta_k(\xi)| \geq C\jp{\xi}^{-M},
		\]
		for all \([\xi]\in\widehat G\) with \(\jp{\xi}\geq R\) and all
		\(1\leq k\leq d_\xi\).
	\end{definition}

	\subsection{Global Hypoellipticity} \
	\label{subsection_GH}
	
	We first consider global hypoellipticity. The question is whether the
	regularity of the datum \(Pu\) implies the same regularity for the distribution
	\(u\). On the Fourier side, this amounts to asking whether rapid decay of the
	coefficients of \(Pu\) can be transferred to the coefficients of \(u\).
	Formula \eqref{cramer} shows that the only possible obstruction to this
	transfer, apart from finitely many low-frequency modes, is the smallness of the
	determinant \(\Delta_k(\xi)\). Condition \emph{(DC)} is precisely the lower
	bound that prevents this division from destroying rapid decay.
	
	\begin{definition}
		Let \(P:\DD(G)\to\DD(G)\) be a continuous \(\R\)-linear operator such that
		\(P(C^\infty(G))\subset C^\infty(G).\) 
		We say that \(P\) is globally hypoelliptic if
		\[
		u\in\DD(G),
		\quad
		Pu\in C^\infty(G)
		\quad\Longrightarrow\quad
		u\in C^\infty(G).
		\]
	\end{definition}
	
	\begin{proposition}\label{firstprop}
		Let \(L:\DD(G)\to\DD(G)\) be a diagonal operator in the sense of
		Definition~\ref{diagonal}, let \(p,q\in\C\) with \(p\neq0\), and let \(P\)
		be the associated Vekua-type operator defined by \eqref{vekua_op}. 
		
		If \(P\)
		satisfies condition \emph{(DC)}, then \(P\) is globally hypoelliptic.
	\end{proposition}
	
	\begin{proof}
		Suppose that \(P\) satisfies condition \emph{(DC)}. Thus, there exist
		constants \(C>0\), \(M>0\), and \(R>0\) such that
		\[
		|\Delta_k(\xi)|
		\geq
		C\jp{\xi}^{-M},
		\]
		for all \([\xi]\in\widehat G\) with \(\jp{\xi}\geq R\) and all
		\(1\leq k\leq d_\xi\). 
		
		Let \(u\in\DD(G)\) and suppose that
		\[
		Pu=f\in C^\infty(G).
		\]
		
		We shall prove that \(u\in C^\infty(G)\) by showing that its Fourier
		coefficients are rapidly decreasing.
		
		Let \(N>0\) be arbitrary. Since the set
		\(\{[\xi]\in\widehat G: \jp{\xi}<R\} \) 
		is finite, we may define
		\[
		C_1(N)
		:=
		\max
		\left\{
		|\widehat u(\xi)_{k\ell}|\jp{\xi}^{N}
		:
		\jp{\xi}<R,\;
		1\leq k,\ell\leq d_\xi
		\right\},
		\]
		whenever this set is nonempty; otherwise, we set \(C_1(N)=0\). 
		
		Then
		\[
		|\widehat u(\xi)_{k\ell}|
		\leq
		C_1(N)\jp{\xi}^{-N}
		\]
		for all \([\xi]\in\widehat G\) with \(\jp{\xi}<R\) and all
		\(1\leq k,\ell\leq d_\xi\).

		Since \(f\in C^\infty(G)\), we also have \(\overline f\in C^\infty(G)\).
		Hence, by the Fourier characterization of smooth functions, there exist
		constants \(C_2(N)>0\) and \(C_3(N)>0\) such that
		\[
		|\widehat f(\xi)_{k\ell}|
		\leq
		C_2(N)\jp{\xi}^{-N-M-K}
		\]
		and
		\[
		|\widehat{\overline f}(\xi)_{k\ell}|
		\leq
		C_3(N)\jp{\xi}^{-N-M},
		\]
		for all \([\xi]\in\widehat G\) and all \(1\leq k,\ell\leq d_\xi\).

		Now let \([\xi]\in\widehat G\) with \(\jp{\xi}\geq R\). By
		\eqref{cramer}, the triangle inequality, condition \emph{(DC)}, and the
		symbol estimate \eqref{condsimbolo}, we obtain
		\begin{align*}
			|\widehat u(\xi)_{k\ell}|
			&\leq
			\frac{1}{|\Delta_k(\xi)|}
			\left[
			|\sigma_k(\xi)-\overline q|\,
			|\widehat f(\xi)_{k\ell}|
			+
			|p|\,|\widehat{\overline f}(\xi)_{k\ell}|
			\right]
			\\
			&\leq
			\frac{1}{C}\jp{\xi}^{M}
			\left[
			\big(C_\sigma\jp{\xi}^{K}+|q|\big)
			C_2(N)\jp{\xi}^{-N-M-K}
			+
			|p|C_3(N)\jp{\xi}^{-N-M}
			\right]
			\\
			&\leq
			\frac{C_\sigma C_2(N)}{C}\jp{\xi}^{-N}
			+
			\frac{|q|C_2(N)}{C}\jp{\xi}^{-N-K}
			+
			\frac{|p|C_3(N)}{C}\jp{\xi}^{-N}
			\\
			&\leq
			\frac{1}{C}
			\left[
			C_\sigma C_2(N)
			+
			|q|C_2(N)
			+
			|p|C_3(N)
			\right]
			\jp{\xi}^{-N}.
		\end{align*}
		In the last inequality we used that \(\jp{\xi}\geq1\), so that
		\(\jp{\xi}^{-N-K}\leq\jp{\xi}^{-N}\).
		
		Therefore, setting
		\[
		\widetilde C(N)
		:=
		\max
		\left\{
		C_1(N),
		\frac{1}{C}
		\left[
		C_\sigma C_2(N)
		+
		|q|C_2(N)
		+
		|p|C_3(N)
		\right]
		\right\},
		\]
		we have
		\[
		|\widehat u(\xi)_{k\ell}|
		\leq
		\widetilde C(N)\jp{\xi}^{-N},
		\]
		for all \([\xi]\in\widehat G\) and all \(1\leq k,\ell\leq d_\xi\).
		
		Since \(N>0\) was arbitrary, the Fourier coefficients of \(u\) are rapidly
		decreasing. By the Fourier characterization of smooth functions on compact
		Lie groups, it follows that \(u\in C^\infty(G)\). Hence \(P\) is globally
		hypoelliptic.
	\end{proof}
	
	Note that if condition \emph{(DC)} holds, then \(\Delta_k(\xi)\) can vanish
	for at most finitely many representation classes \([\xi]\in\widehat G\) and
	indices \(1\leq k\leq d_\xi\). We now show that, under a non-self-duality
	assumption on the relevant representations, this finiteness property is
	necessary for global hypoellipticity.
	
	We introduce the zero set
	\begin{equation}\label{setZ}
		\mathcal Z
		:=
		\left\{
		[\xi]\in\widehat G:
		\Delta_k(\xi)=0
		\text{ for some } 1\leq k\leq d_\xi
		\right\}.
	\end{equation}
	
	The next result isolates the obstruction produced by infinitely many zeros of
	the determinant occurring along non-self-dual representation classes.
	
	\begin{proposition}\label{Zinfinite}
		Suppose that \(\mathcal Z\) contains infinitely many non-self-dual
		representation classes. Then \(P\) is not globally hypoelliptic.
	\end{proposition}
	
	\begin{proof}
		Since \(\mathcal Z\) contains infinitely many non-self-dual representation
		classes, we may choose a sequence of mutually distinct classes
		\([\xi_n]\in\mathcal Z\) and indices \(1\leq k_n\leq d_{\xi_n}\) such that
		\[
		\Delta_{k_n}(\xi_n)=0,
		\qquad n\in\N,
		\]
		and
		\[
		[\xi_n]\neq[\overline{\xi_n}],
		\qquad n\in\N.
		\]
		
		Passing to a subsequence if necessary, we may also assume that
		\[
		[\xi_n]\neq[\overline{\xi_m}]
		\qquad
		\text{for all } n,m\in\N.
		\]
		
		Define a family of matrices
		\[
		x(\xi)=\big(x(\xi)_{k\ell}\big)_{1\leq k,\ell\leq d_\xi},
		\qquad [\xi]\in\widehat G,
		\]
		by
		\[
		x(\xi)_{k\ell}
		=
		\begin{cases}
			\sigma_{k_n}(\xi_n)-\overline q,
			& \text{if } [\xi]=[\xi_n] \text{ and } k=\ell=k_n, \\[0.3em]
			p,
			& \text{if } [\xi]=[\overline{\xi_n}] \text{ and } k=\ell=k_n, \\[0.3em]
			0,
			& \text{otherwise}.
		\end{cases}
		\]
		
		We first show that this family defines a distribution. By the symbol estimate
		\eqref{condsimbolo} and the triangle inequality,
		\[
		|x(\xi_n)_{k_nk_n}|
		=
		|\sigma_{k_n}(\xi_n)-\overline q|
		\leq
		C_\sigma\jp{\xi_n}^{K}+|q|.
		\]
		
		Since \(\jp{\xi_n}\geq1\), it follows that
		\[
		|x(\xi_n)_{k_nk_n}|
		\leq
		(C_\sigma+|q|)\jp{\xi_n}^{K}.
		\]
		
		Moreover, using \(\jp{\overline{\xi_n}}=\jp{\xi_n}\geq1\), we obtain
		\[
		|x(\overline{\xi_n})_{k_nk_n}|
		=
		|p|
		\leq
		|p|\jp{\overline{\xi_n}}^{K}.
		\]
		
		Combining these estimates with the definition of \(x\), and setting
		\[
		C_0:=\max\{C_\sigma+|q|,|p|\},
		\]
		we obtain
		\[
		|x(\xi)_{k\ell}|
		\leq
		C_0\jp{\xi}^{K},
		\]
		for all \([\xi]\in\widehat G\) and all \(1\leq k,\ell\leq d_\xi\).
		
		By the Fourier characterization of distributions, there exists
		\(u\in\DD(G)\) such that
		\[
		\widehat u(\xi)_{k\ell}=x(\xi)_{k\ell}.
		\]
		
		We claim that \(u\) is not smooth. Indeed, since the classes
		\([\overline{\xi_n}]\) are mutually distinct and only finitely many
		representation classes satisfy \(\jp{\xi}\leq R\) for each \(R>0\), we have
		\[
		\jp{\overline{\xi_n}}=\jp{\xi_n}\to\infty.
		\]
		
		However,
		\[
		|\widehat u(\overline{\xi_n})_{k_nk_n}|
		=
		|p|
		\neq0,
		\qquad n\in\N.
		\]
		
		Thus the Fourier coefficients of \(u\) are not rapidly decreasing, and
		therefore \(u\notin C^\infty(G)\).
		
		It remains to prove that \(Pu=0\). If
		\[
		[\xi]\notin \{[\xi_n],[\overline{\xi_n}]:n\in\N\},
		\]
		then both \(\widehat u(\xi)\) and \(\widehat{\overline u}(\xi)\) vanish, and
		therefore
		\[
		\widehat{Pu}(\xi)=0.
		\]
		
		We now check the classes \([\xi_n]\) and \([\overline{\xi_n}]\). Using
		\(\widehat{\overline u}(\xi)=\overline{\widehat u(\overline\xi)}\), we obtain
		\begin{align*}
			\widehat{Pu}(\xi_n)_{k_nk_n}
			&=
			(\sigma_{k_n}(\xi_n)-q)
			\widehat u(\xi_n)_{k_nk_n}
			-
			p\,\widehat{\overline u}(\xi_n)_{k_nk_n}
			\\
			&=
			(\sigma_{k_n}(\xi_n)-q)
			(\sigma_{k_n}(\xi_n)-\overline q)
			-
			p\,\overline{\widehat u(\overline{\xi_n})_{k_nk_n}}
			\\
			&=
			(\sigma_{k_n}(\xi_n)-q)
			(\sigma_{k_n}(\xi_n)-\overline q)
			-
			p\,\overline p
			\\
			&=
			\Delta_{k_n}(\xi_n)
			=
			0.
		\end{align*}
		
		For the conjugate class, using
		\[
		\sigma_{k_n}(\overline{\xi_n})
		=
		\overline{\sigma_{k_n}(\xi_n)},
		\]
		we have
		\begin{align*}
			\widehat{Pu}(\overline{\xi_n})_{k_nk_n}
			&=
			(\sigma_{k_n}(\overline{\xi_n})-q)
			\widehat u(\overline{\xi_n})_{k_nk_n}
			-
			p\,\widehat{\overline u}(\overline{\xi_n})_{k_nk_n}
			\\
			&=
			(\overline{\sigma_{k_n}(\xi_n)}-q)p
			-
			p\,\overline{\widehat u(\xi_n)_{k_nk_n}}
			\\
			&=
			p(\overline{\sigma_{k_n}(\xi_n)}-q)
			-
			p\,\overline{\sigma_{k_n}(\xi_n)-\overline q}
			\\
			&=
			p(\overline{\sigma_{k_n}(\xi_n)}-q)
			-
			p(\overline{\sigma_{k_n}(\xi_n)}-q)
			=
			0.
		\end{align*}
		
		All remaining matrix entries in the classes \([\xi_n]\) and
		\([\overline{\xi_n}]\) vanish by construction, and the corresponding entries
		of \(\widehat{\overline u}\) vanish as well. Hence
		\[
		\widehat{Pu}(\xi)=0
		\qquad
		\text{for every }[\xi]\in\widehat G.
		\]
		
		Therefore
		\[
		Pu=0\in C^\infty(G),
		\qquad
		u\in\DD(G)\setminus C^\infty(G).
		\]
		
		This shows that \(P\) is not globally hypoelliptic.
	\end{proof}
	
	For the full necessity of condition \emph{(DC)}, we need the non-self-dual
	mechanism to be available at all sufficiently high frequencies. We therefore
	assume that only finitely many irreducible representations are self-dual.
	
	\begin{proposition}\label{DCnecessary}
		Suppose that, outside a finite subset of \(\widehat G\), all irreducible
		representations are non-self-dual. If the Vekua-type operator \(P\) is
		globally hypoelliptic, then condition \emph{(DC)} holds.
	\end{proposition}
	
	\begin{proof}
		We argue by contraposition. Suppose that condition \emph{(DC)} does not
		hold.
		
		If the zero set \(\mathcal Z\) is infinite, then, since only finitely many
		irreducible representations are self-dual, \(\mathcal Z\) contains infinitely
		many non-self-dual representation classes. Hence
		Proposition~\ref{Zinfinite} implies that \(P\) is not globally
		hypoelliptic. Thus, we may assume that \(\mathcal Z\) is finite.
		
		Since condition \emph{(DC)} fails and \(\mathcal Z\) is finite, there exist
		a sequence of mutually distinct representation classes
		\([\xi_n]\in\widehat G\) and indices \(1\leq k_n\leq d_{\xi_n}\) such that
		\[
		\jp{\xi_n}>n,
		\qquad
		0<|\Delta_{k_n}(\xi_n)|<\jp{\xi_n}^{-n},
		\qquad n\in\N.
		\]
		
		Moreover, since only finitely many irreducible representations are
		self-dual, we may pass to a subsequence and assume that all classes
		\([\xi_n]\) are non-self-dual. By choosing at most one class from each
		conjugate pair, we may also assume that
		\[
		[\xi_n]\neq[\overline{\xi_m}]
		\qquad
		\text{for all } n,m\in\N.
		\]
		
		Define a family of matrices
		\[
		x(\xi)=\big(x(\xi)_{k\ell}\big)_{1\leq k,\ell\leq d_\xi},
		\qquad [\xi]\in\widehat G,
		\]
		by
		\[
		x(\xi)_{k\ell}
		:=
		\begin{cases}
			\Delta_{k_n}(\xi_n),
			& \text{if } [\xi]=[\xi_n] \text{ and } k=\ell=k_n, \\[0.3em]
			\Delta_{k_n}(\overline{\xi_n}),
			& \text{if } [\xi]=[\overline{\xi_n}] \text{ and } k=\ell=k_n, \\[0.3em]
			0,
			& \text{otherwise}.
		\end{cases}
		\]
		
		Since
		\[
		|\Delta_{k_n}(\xi_n)|<\jp{\xi_n}^{-n}
		\]
		and
		\[
		\Delta_{k_n}(\overline{\xi_n})
		=
		\overline{\Delta_{k_n}(\xi_n)},
		\qquad
		\jp{\overline{\xi_n}}=\jp{\xi_n},
		\]
		the family \((x(\xi))_{[\xi]\in\widehat G}\) is rapidly decreasing.
		Therefore, by the Fourier characterization of smooth functions, there
		exists \(f\in C^\infty(G)\) such that
		\[
		\widehat f(\xi)_{k\ell}=x(\xi)_{k\ell}.
		\]
		
		We now choose a constant \(c\in\C\setminus\{0\}\). Let
		\[
		s:=\frac{\overline c}{c},
		\qquad |s|=1,
		\]
		and define
		\[
		\alpha_c:=\overline q-p\frac{\overline c}{c}
		=
		\overline q-ps.
		\]
		
		We choose \(s\) on the unit circle such that
		\begin{equation}\label{choice_s}
			(\alpha_c-q)(\alpha_c-\overline q)-|p|^2\neq0.
		\end{equation}
		
		Such a choice is always possible. Indeed, the expression
		\[
		(\overline q-ps-q)(\overline q-ps-\overline q)-|p|^2
		\]
		is a nonzero polynomial in \(s\), because \(p\neq0\), and hence it has only
		finitely many zeros on the unit circle. After choosing such an \(s\), we fix
		any \(c\in\C\setminus\{0\}\) satisfying \(\overline c/c=s\).
		
		For this fixed choice of \(c\), define a family of matrices
		\[
		y(\xi)=\big(y(\xi)_{k\ell}\big)_{1\leq k,\ell\leq d_\xi},
		\qquad [\xi]\in\widehat G,
		\]
		by
		\[
		y(\xi)_{k\ell}
		:=
		\begin{cases}
			c\big(\sigma_{k_n}(\xi)-\overline q\big)+p\overline c,
			& \text{if } \big([\xi]=[\xi_n]\text{ or }[\xi]=[\overline{\xi_n}]\big)
			\text{ and } k=\ell=k_n, \\[0.3em]
			0,
			& \text{otherwise}.
		\end{cases}
		\]
		
		By the symbol estimate \eqref{condsimbolo}, we have
		\[
		|y(\xi)_{k\ell}|
		\leq
		|c|\big(|\sigma_{k_n}(\xi)|+|q|\big)+|p|\,|c|
		\leq
		|c|(C_\sigma+|q|+|p|)\jp{\xi}^{K}.
		\]
		
		Hence \((y(\xi))_{[\xi]\in\widehat G}\) has polynomial growth and therefore
		defines a distribution \(u_c\in\DD(G)\) such that
		\[
		\widehat{u_c}(\xi)_{k\ell}=y(\xi)_{k\ell}.
		\]
		
		We claim that
		\[
		Pu_c=cf.
		\]
		It is enough to verify this on the classes \([\xi_n]\) and
		\([\overline{\xi_n}]\), since all other Fourier coefficients vanish by
		construction.
		
		For \([\xi_n]\), using \eqref{subsystem}, we obtain
		\begin{align*}
			\widehat{Pu_c}(\xi_n)_{k_nk_n}
			&=
			(\sigma_{k_n}(\xi_n)-q)\widehat{u_c}(\xi_n)_{k_nk_n}
			-
			p\,\widehat{\overline{u_c}}(\xi_n)_{k_nk_n}
			\\
			&=
			(\sigma_{k_n}(\xi_n)-q)
			\big(c(\sigma_{k_n}(\xi_n)-\overline q)+p\overline c\big)
			\\
			&\quad
			-
			p\,
			\overline{
				c(\sigma_{k_n}(\overline{\xi_n})-\overline q)+p\overline c
			}.
		\end{align*}
		
		Since
		\[
		\sigma_{k_n}(\overline{\xi_n})
		=
		\overline{\sigma_{k_n}(\xi_n)},
		\]
		the last expression becomes
		\begin{align*}
			\widehat{Pu_c}(\xi_n)_{k_nk_n}
			&=
			(\sigma_{k_n}(\xi_n)-q)
			\big(c(\sigma_{k_n}(\xi_n)-\overline q)+p\overline c\big)
			\\
			&\quad
			-
			p\big(
			\overline c(\sigma_{k_n}(\xi_n)-q)+\overline p c
			\big)
			\\
			&=
			c\big[
			(\sigma_{k_n}(\xi_n)-q)(\sigma_{k_n}(\xi_n)-\overline q)
			-|p|^2
			\big]
			\\
			&=
			c\Delta_{k_n}(\xi_n)
			=
			\widehat{cf}(\xi_n)_{k_nk_n}.
		\end{align*}
		
		For the conjugate class, the same computation gives
		\begin{align*}
			\widehat{Pu_c}(\overline{\xi_n})_{k_nk_n}
			&=
			c\big[
			(\sigma_{k_n}(\overline{\xi_n})-q)
			(\sigma_{k_n}(\overline{\xi_n})-\overline q)
			-|p|^2
			\big]
			\\
			&=
			c\Delta_{k_n}(\overline{\xi_n})
			=
			\widehat{cf}(\overline{\xi_n})_{k_nk_n}.
		\end{align*}
		
		Therefore
		\[
		Pu_c=cf\in C^\infty(G).
		\]
		
		We now show that \(u_c\) is not smooth. Suppose, by contradiction, that
		\(u_c\in C^\infty(G)\). Then its Fourier coefficients are rapidly
		decreasing. In particular,
		\[
		\widehat{u_c}(\xi_n)_{k_nk_n}
		=
		c\big(\sigma_{k_n}(\xi_n)-\overline q\big)+p\overline c
		\]
		must tend to zero as \(n\to\infty\). Since \(c\neq0\), this implies
		\[
		\sigma_{k_n}(\xi_n)
		\longrightarrow
		\overline q-p\frac{\overline c}{c}
		=
		\alpha_c.
		\]
		
		On the other hand,
		\[
		\Delta_{k_n}(\xi_n)
		=
		(\sigma_{k_n}(\xi_n)-q)(\sigma_{k_n}(\xi_n)-\overline q)-|p|^2
		\longrightarrow 0.
		\]
		
		Passing to the limit, we get
		\[
		(\alpha_c-q)(\alpha_c-\overline q)-|p|^2=0,
		\]
		which contradicts the choice of \(c\) in \eqref{choice_s}.
		
		Hence
		\[
		u_c\in\DD(G)\setminus C^\infty(G),
		\]
		while
		\[
		Pu_c=cf\in C^\infty(G).
		\]
		
		Therefore \(P\) is not globally hypoelliptic.
	\end{proof}
	
	We can now state the main result of this subsection. Combining the sufficiency
	given by Proposition~\ref{firstprop} with the necessity established in
	Proposition~\ref{DCnecessary}, we obtain the following characterization.
	
	\begin{theorem}\label{main_GH}
		Let \(L:\DD(G)\to\DD(G)\) be a diagonal operator in the sense of
		Definition~\ref{diagonal}, and let \(p,q\in\C\) with \(p\neq0\). Consider
		the associated Vekua-type operator
		\[
		Pu=Lu-q\,u-p\,\overline u.
		\]
		Suppose that, outside a finite subset of \(\widehat G\), all irreducible
		representations of \(G\) are non-self-dual. Then \(P\) is globally
		hypoelliptic if and only if \(P\) satisfies condition \emph{(DC)}.
	\end{theorem}

	\subsection{Global Solvability} \
	\label{Subsection_GS}
		
	We now turn to global solvability. Throughout this subsection, we assume that,
	outside a finite subset of \(\widehat G\), all irreducible representations are
	non-self-dual. The finitely many self-dual exceptional classes, if present, give
	rise only to finite-dimensional algebraic systems. Their solvability conditions
	will be included in the admissible space defined below. Consequently, these
	exceptional classes do not affect the asymptotic condition \emph{(DC')}, which
	is imposed only to control the high-frequency behavior of the inverse systems.
	
	Global solvability asks whether the equation
	\[
	Pu=f
	\]
	can be solved for a prescribed datum \(f\). On the Fourier side, this question
	is again governed by the determinant \(\Delta_k(\xi)\), but now its zeros play
	an additional role: they impose algebraic compatibility conditions on the datum.
	
	Indeed, let \(u,f\in\DD(G)\) satisfy \(Pu=f\). By \eqref{cramer}, for every
	\([\xi]\in\widehat G\) and every \(1\leq k,\ell\leq d_\xi\), we have
	\[
	\Delta_k(\xi)\widehat u(\xi)_{k\ell}
	=
	(\sigma_k(\xi)-\overline q)\widehat f(\xi)_{k\ell}
	+
	p\,\overline{\widehat f(\overline\xi)_{k\ell}}.
	\]
	Therefore, whenever \(\Delta_k(\xi)=0\), the right-hand side must vanish. Hence
	the datum \(f\) must satisfy the compatibility condition
	\begin{equation}\label{compatibility_condition}
		(\sigma_k(\xi)-\overline q)\widehat f(\xi)_{k\ell}
		+
		p\,\overline{\widehat f(\overline\xi)_{k\ell}}
		=
		0.
	\end{equation}
	
	\begin{definition}\label{setA}
		We denote by \(\mathcal A\) the space of all distributions \(f\in\DD(G)\)
		satisfying the following compatibility conditions:
		\begin{enumerate}
			\item[(A1)] For every non-self-dual class \([\xi]\in\widehat G\), every
			\(1\leq k,\ell\leq d_\xi\), and every zero of the determinant,
			\[
			\Delta_k(\xi)=0,
			\]
			one has
			\[
			(\sigma_k(\xi)-\overline q)\widehat f(\xi)_{k\ell}
			+
			p\,\overline{\widehat f(\overline\xi)_{k\ell}}
			=
			0.
			\]
			
			\item[(A2)] On the finitely many self-dual exceptional classes, the corresponding
			Fourier coefficients of \(f\) belong to the range of the finite-dimensional
			algebraic system induced by \(P\) on those classes.
		\end{enumerate}
	\end{definition}
	
	\begin{definition}
		We say that the Vekua-type operator \(P\) is globally solvable if
		\[
		P(\DD(G))=\mathcal A.
		\]
	\end{definition}
	
	We now explain how to construct a solution on the Fourier side. Let
	\(f\in\mathcal A\). If \(\Delta_k(\xi)\neq0\), then any solution of \(Pu=f\)
	must satisfy, by \eqref{cramer},
	\[
	\widehat u(\xi)_{k\ell}
	=
	\frac{
		(\sigma_k(\xi)-\overline q)\widehat f(\xi)_{k\ell}
		+
		p\,\overline{\widehat f(\overline\xi)_{k\ell}}
	}{
		\Delta_k(\xi)
	}.
	\]
	This formula determines the Fourier coefficients of the solution away from the
	zero set of \(\Delta_k\).
	
	On the zero set, the system is not invertible. In the non-self-dual classes, we
	choose the coefficients in conjugate pairs. Since
	\[
	\Delta_k(\overline\xi)=\overline{\Delta_k(\xi)},
	\]
	the zero set is invariant under conjugation. For each \(k\), outside the
	finitely many self-dual exceptional classes, we choose a disjoint decomposition
	\[
	\{[\xi]\in\widehat G:\Delta_k(\xi)=0\}
	=
	\mathcal Z_k^+\cup\mathcal Z_k^-,
	\]
	such that
	\[
	[\xi]\in\mathcal Z_k^+
	\quad\Longleftrightarrow\quad
	[\overline\xi]\in\mathcal Z_k^-.
	\]
	
	We now define a family of matrices
	\[
	x(\xi)=\big(x(\xi)_{k\ell}\big)_{1\leq k,\ell\leq d_\xi},
	\qquad [\xi]\in\widehat G.
	\]

	If \(\Delta_k(\xi)\neq0\), we set
	\begin{equation}\label{nonzero_assignment}
		x(\xi)_{k\ell}
		:=
		\frac{
			(\sigma_k(\xi)-\overline q)\widehat f(\xi)_{k\ell}
			+
			p\,\widehat{\overline f}(\xi)_{k\ell}
		}{
			\Delta_k(\xi)
		}.
	\end{equation}
	
	If \(\Delta_k(\xi)=0\) and \([\xi]\in\mathcal Z_k^+\), we set
	\begin{equation}\label{zero_assignment}
		x(\xi)_{k\ell}:=0,
		\qquad
		x(\overline\xi)_{k\ell}
		:=
		-\frac{1}{\overline p}\,
		\overline{\widehat f(\xi)_{k\ell}}.
	\end{equation}
	
	On the finitely many exceptional self-dual classes, the coefficients are chosen
	as any solution of the corresponding finite-dimensional algebraic system, whose
	solvability is part of the admissibility condition defining \(\mathcal A\).
	
	Thus the family \(x(\xi)\) is defined for every representation class and every
	matrix entry. The only remaining issue is to guarantee that this family has at
	most polynomial growth. This is precisely the role of the following
	Diophantine-type condition away from the zero set.
	
	\begin{definition}\label{DCprime}
		Let \(L:\DD(G)\to\DD(G)\) be a diagonal operator in the sense of
		Definition~\ref{diagonal}, let \(p,q\in\C\) with \(p\neq0\), and let \(P\)
		be the associated Vekua-type operator defined by \eqref{vekua_op}. We say
		that \(P\) satisfies condition \emph{(DC')} if there exist constants
		\(C>0\) and \(M>0\) such that
		\begin{equation}\label{DC_prime}
			|\Delta_k(\xi)|
			\geq
			C\jp{\xi}^{-M},
		\end{equation}
		for every \([\xi]\in\widehat G\) and every \(1\leq k\leq d_\xi\) such that
		\(\Delta_k(\xi)\neq0\).
	\end{definition}	
		
	\begin{proposition}\label{GSsufficient}
		If \(P\) satisfies condition \emph{(DC')}, then \(P\) is globally solvable.
	\end{proposition}
	
	\begin{proof}
		Assume that condition \emph{(DC')} holds, and let \(C>0\) and \(M>0\) be as
		in Definition~\ref{DCprime}. Let \(f\in\mathcal A\). We shall construct
		\(u\in\DD(G)\) such that \(Pu=f\).
		
		We use the family of matrices
		\[
		x(\xi)=\big(x(\xi)_{k\ell}\big)_{1\leq k,\ell\leq d_\xi},
		\qquad [\xi]\in\widehat G,
		\]
		defined by \eqref{nonzero_assignment} and \eqref{zero_assignment}, together
		with the prescribed solutions on the finitely many self-dual exceptional
		classes.
		
		Since \(f\in\DD(G)\), there exist constants \(C_f>0\) and \(N_f>0\) such that
		\[
		|\widehat f(\xi)_{k\ell}|
		\leq
		C_f\jp{\xi}^{N_f},
		\]
		for all \([\xi]\in\widehat G\) and all \(1\leq k,\ell\leq d_\xi\).
		
		Suppose first that \(\Delta_k(\xi)=0\) and \([\xi]\in\mathcal Z_k^+\). Then
		\[
		x(\xi)_{k\ell}=0,
		\qquad
		x(\overline\xi)_{k\ell}
		=
		-\frac{1}{\overline p}\,
		\overline{\widehat f(\xi)_{k\ell}}.
		\]
		
		Hence, using \(\jp{\overline\xi}=\jp{\xi}\), we obtain
		\[
		|x(\overline\xi)_{k\ell}|
		\leq
		\frac{C_f}{|p|}\jp{\overline\xi}^{N_f}.
		\]
		
		Now suppose that \(\Delta_k(\xi)\neq0\). By condition \emph{(DC')}, the
		symbol estimate \eqref{condsimbolo}, and the polynomial growth of
		\(\widehat f\), we have
		\begin{align*}
			|x(\xi)_{k\ell}|
			&\leq
			\frac{1}{|\Delta_k(\xi)|}
			\left[
			(|\sigma_k(\xi)|+|q|)
			|\widehat f(\xi)_{k\ell}|
			+
			|p|\,|\widehat f(\overline\xi)_{k\ell}|
			\right]
			\\
			&\leq
			\frac{1}{C}
			\jp{\xi}^{M}
			\left[
			(C_\sigma\jp{\xi}^{K}+|q|)
			C_f\jp{\xi}^{N_f}
			+
			|p|C_f\jp{\xi}^{N_f}
			\right]
			\\
			&\leq
			\frac{C_f}{C}
			(C_\sigma+|q|+|p|)
			\jp{\xi}^{M+K+N_f}.
		\end{align*}
		
		On the finitely many self-dual exceptional classes, the coefficients are chosen
		by solving the corresponding finite-dimensional algebraic systems. Since there
		are only finitely many such classes, their contribution can be absorbed into
		the constant in the polynomial estimate.
		
		Therefore, there exists a constant \(\widetilde C>0\) such that
		\[
		|x(\xi)_{k\ell}|
		\leq
		\widetilde C\jp{\xi}^{M+K+N_f},
		\]
		for all \([\xi]\in\widehat G\) and all \(1\leq k,\ell\leq d_\xi\). By the
		Fourier characterization of distributions, there exists \(u\in\DD(G)\) such
		that
		\[
		\widehat u(\xi)_{k\ell}=x(\xi)_{k\ell}.
		\]
		
		It remains to verify that \(Pu=f\). We first consider the case \(\Delta_k(\xi)\neq0\). By the definition of
		\(x(\xi)_{k\ell}\), we have
		\[
		\widehat u(\xi)_{k\ell}
		=
		\frac{
			(\sigma_k(\xi)-\overline q)\widehat f(\xi)_{k\ell}
			+
			p\,\widehat{\overline f}(\xi)_{k\ell}
		}{
			\Delta_k(\xi)
		}.
		\]
		
		Moreover, using the definition of \(x(\overline\xi)_{k\ell}\), together with
		\[
		\sigma_k(\overline\xi)=\overline{\sigma_k(\xi)},
		\qquad
		\Delta_k(\overline\xi)=\overline{\Delta_k(\xi)},
		\]
		we obtain
		\[
		\widehat{\overline u}(\xi)_{k\ell}
		=
		\frac{
			(\sigma_k(\xi)-q)\widehat{\overline f}(\xi)_{k\ell}
			+
			\overline p\,\widehat f(\xi)_{k\ell}
		}{
			\Delta_k(\xi)
		}.
		\]
		
		Therefore,
		\begin{align*}
			\widehat{Pu}(\xi)_{k\ell}
			&=
			(\sigma_k(\xi)-q)\widehat u(\xi)_{k\ell}
			-
			p\,\widehat{\overline u}(\xi)_{k\ell}
			\\
			&=
			\frac{
				(\sigma_k(\xi)-q)
				\big[
				(\sigma_k(\xi)-\overline q)\widehat f(\xi)_{k\ell}
				+
				p\,\widehat{\overline f}(\xi)_{k\ell}
				\big]
			}{
				\Delta_k(\xi)
			}
			\\
			&\quad
			-
			\frac{
				p\big[
				(\sigma_k(\xi)-q)\widehat{\overline f}(\xi)_{k\ell}
				+
				\overline p\,\widehat f(\xi)_{k\ell}
				\big]
			}{
				\Delta_k(\xi)
			}
			\\
			&=
			\frac{
				\big[
				(\sigma_k(\xi)-q)(\sigma_k(\xi)-\overline q)
				-|p|^2
				\big]
				\widehat f(\xi)_{k\ell}
			}{
				\Delta_k(\xi)
			}
			\\
			&=
			\widehat f(\xi)_{k\ell}.
		\end{align*}

		Hence
		\[
		\widehat{Pu}(\xi)_{k\ell}
		=
		\widehat f(\xi)_{k\ell}
		\]
		whenever \(\Delta_k(\xi)\neq0\).
		
		Now suppose that \(\Delta_k(\xi)=0\) and, without loss of generality, that
		\([\xi]\in\mathcal Z_k^+\). Then
		\[
		\widehat u(\xi)_{k\ell}=0,
		\qquad
		\widehat u(\overline\xi)_{k\ell}
		=
		-\frac{1}{\overline p}\,
		\overline{\widehat f(\xi)_{k\ell}}.
		\]
		
		Using
		\[
		\widehat{\overline u}(\xi)_{k\ell}
		=
		\overline{\widehat u(\overline\xi)_{k\ell}},
		\]
		we obtain
		\begin{align*}
			\widehat{Pu}(\xi)_{k\ell}
			&=
			(\sigma_k(\xi)-q)\widehat u(\xi)_{k\ell}
			-
			p\,\widehat{\overline u}(\xi)_{k\ell}
			\\
			&=
			-p\,
			\overline{
				-\frac{1}{\overline p}
				\overline{\widehat f(\xi)_{k\ell}}
			}
			\\
			&=
			\widehat f(\xi)_{k\ell}.
		\end{align*}
		
		For the conjugate class, using
		\[
		\sigma_k(\overline\xi)=\overline{\sigma_k(\xi)}
		\]
		and \(\widehat u(\xi)_{k\ell}=0\), we get
		\begin{align*}
			\widehat{Pu}(\overline\xi)_{k\ell}
			&=
			(\sigma_k(\overline\xi)-q)
			\widehat u(\overline\xi)_{k\ell}
			-
			p\,\widehat{\overline u}(\overline\xi)_{k\ell}
			\\
			&=
			(\overline{\sigma_k(\xi)}-q)
			\left(
			-\frac{1}{\overline p}
			\overline{\widehat f(\xi)_{k\ell}}
			\right).
		\end{align*}

		Since \(f\in\mathcal A\), the compatibility condition gives
		\[
		(\sigma_k(\xi)-\overline q)\widehat f(\xi)_{k\ell}
		=
		-p\,\overline{\widehat f(\overline\xi)_{k\ell}}.
		\]
		
		Taking complex conjugates, we obtain
		\[
		(\overline{\sigma_k(\xi)}-q)
		\overline{\widehat f(\xi)_{k\ell}}
		=
		-\overline p\,\widehat f(\overline\xi)_{k\ell}.
		\]
		
		Therefore,
		\[
		\widehat{Pu}(\overline\xi)_{k\ell}
		=
		\widehat f(\overline\xi)_{k\ell}.
		\]
		
		On the finitely many self-dual exceptional classes, the equality
		\[
		\widehat{Pu}(\xi)=\widehat f(\xi)
		\]
		holds by the way the coefficients were chosen from the corresponding
		finite-dimensional algebraic systems.
		
		Hence
		\[
		\widehat{Pu}(\xi)=\widehat f(\xi)
		\qquad
		\text{for every }[\xi]\in\widehat G.
		\]
		Consequently, \(Pu=f\). Since \(f\in\mathcal A\) was arbitrary, \(P\) is
		globally solvable.
	\end{proof}
	
	We can now state the main global solvability result of this subsection. It shows
	that, under the asymptotic non-self-duality assumption, condition
	\emph{(DC')} is not only sufficient but also necessary for solvability.
	
	\begin{theorem}\label{main_GS}
		Let \(L:\DD(G)\to\DD(G)\) be a diagonal operator in the sense of
		Definition~\ref{diagonal}, and let \(p,q\in\C\) with \(p\neq0\). Consider
		the associated Vekua-type operator
		\[
		Pu=Lu-q\,u-p\,\overline u.
		\]
		Suppose that, outside a finite subset of \(\widehat G\), all irreducible
		representations of \(G\) are non-self-dual. Then \(P\) is globally solvable
		if and only if \(P\) satisfies condition \emph{(DC')}.
	\end{theorem}
	
	\begin{proof}
		By Proposition~\ref{GSsufficient}, condition \emph{(DC')} implies global
		solvability.
		
		Conversely, suppose that condition \emph{(DC')} does not hold. We shall
		construct an admissible datum for which no distributional solution exists.
		
		Since \emph{(DC')} fails, there exist a sequence of representation classes
		\([\xi_n]\in\widehat G\) and indices \(1\leq k_n\leq d_{\xi_n}\) such that
		\[
		\jp{\xi_n}>n,
		\qquad
		0<
		|\Delta_{k_n}(\xi_n)|
		<
		\jp{\xi_n}^{-n},
		\qquad n\in\N.
		\]

		Since only finitely many irreducible representations are self-dual, we may
		discard finitely many terms and assume that all classes \([\xi_n]\) are
		non-self-dual. Passing to a subsequence if necessary, we may also assume that
		\[
		[\xi_n]\neq[\xi_m],
		\qquad
		[\xi_n]\neq[\overline{\xi_m}],
		\qquad
		m\neq n.
		\]
		
		Define a family of matrices
		\[
		z(\xi)=\big(z(\xi)_{k\ell}\big)_{1\leq k,\ell\leq d_\xi},
		\qquad [\xi]\in\widehat G,
		\]
		by
		\[
		z(\xi)_{k\ell}
		:=
		\begin{cases}
			\overline p^{-1},
			& \text{if }[\xi]=[\overline{\xi_n}]
			\text{ and }k=\ell=k_n, \\[0.3em]
			0,
			& \text{otherwise}.
		\end{cases}
		\]
		
		Since this family is bounded, it has polynomial growth. Hence, by the
		Fourier characterization of distributions, there exists \(f\in\DD(G)\) such
		that
		\[
		\widehat f(\xi)_{k\ell}=z(\xi)_{k\ell}.
		\]
		
		We claim that \(f\in\mathcal A\). First, \(f\) has no nonzero Fourier
		coefficients on the finitely many self-dual exceptional classes, after
		discarding finitely many terms from the sequence. Hence the finite-dimensional
		compatibility conditions on those classes are satisfied.
		
		Now let \([\eta]\in\widehat G\) be a non-self-dual class such that
		\(\Delta_j(\eta)=0\) for some \(1\leq j\leq d_\eta\). We must verify the
		compatibility condition
		\[
		(\sigma_j(\eta)-\overline q)\widehat f(\eta)_{j\ell}
		+
		p\,\overline{\widehat f(\overline\eta)_{j\ell}}
		=
		0.
		\]
		
		The first term can be nonzero only if \([\eta]=[\overline{\xi_n}]\) and
		\(j=\ell=k_n\) for some \(n\). But this is impossible, because
		\[
		\Delta_{k_n}(\overline{\xi_n})
		=
		\overline{\Delta_{k_n}(\xi_n)}
		\neq0.
		\]
		
		The second term can be nonzero only if
		\([\overline\eta]=[\overline{\xi_n}]\), equivalently
		\([\eta]=[\xi_n]\), and \(j=\ell=k_n\). This is also impossible when
		\(\Delta_j(\eta)=0\), since
		\[
		\Delta_{k_n}(\xi_n)\neq0.
		\]
		Therefore both terms vanish at every zero of the determinant, and so
		\(f\in\mathcal A\).
		
		We now prove that there is no \(u\in\DD(G)\) satisfying \(Pu=f\). Suppose, by
		contradiction, that such a distribution \(u\) exists. Applying \eqref{cramer}
		at the class \([\xi_n]\) and the entry \((k_n,k_n)\), we get
		\begin{align*}
			\Delta_{k_n}(\xi_n)
			\widehat u(\xi_n)_{k_nk_n}
			&=
			(\sigma_{k_n}(\xi_n)-\overline q)
			\widehat f(\xi_n)_{k_nk_n}
			+
			p\,\overline{\widehat f(\overline{\xi_n})_{k_nk_n}}
			\\
			&=
			0+
			p\,\overline{\overline p^{-1}}
			\\
			&=
			1.
		\end{align*}
		
		Consequently,
		\[
		|\widehat u(\xi_n)_{k_nk_n}|
		=
		\frac{1}{|\Delta_{k_n}(\xi_n)|}
		>
		\jp{\xi_n}^{n}.
		\]
		
		Since \(\jp{\xi_n}\to\infty\), this growth is faster than any polynomial.
		This contradicts the Fourier characterization of distributions on compact Lie
		groups, according to which the Fourier coefficients of a distribution have at
		most polynomial growth.
		
		Thus no distributional solution exists for the admissible datum
		\(f\in\mathcal A\). Therefore \(P\) is not globally solvable. This proves the
		converse implication and completes the proof.
	\end{proof}
	
	We point out that the proof of the sufficiency part gives slightly more than
	distributional solvability. If the admissible datum \(f\) is smooth, then the
	constructed solution is also smooth. Indeed, away from the zero set of
	\(\Delta_k(\xi)\), condition \emph{(DC')} produces only a polynomial loss in the
	Fourier coefficients, while on the zero set the coefficients of the solution are
	either set equal to zero or are given by constant multiples of the Fourier
	coefficients of \(f\). Thus rapid decay is preserved.
	
	\begin{corollary}\label{smooth_GS_general}
		Under the assumptions of Theorem~\ref{main_GS}, the operator \(P\) is
		globally solvable in the smooth category. More precisely,
		\[
		P(C^\infty(G))=\mathcal A\cap C^\infty(G).
		\]
	\end{corollary}
	
	\begin{proof}
		The inclusion
		\[
		P(C^\infty(G))\subset \mathcal A\cap C^\infty(G)
		\]
		follows from the definition of \(\mathcal A\) and from the fact that \(P\)
		preserves smooth functions. Conversely, if \(f\in\mathcal A\cap C^\infty(G)\),
		the construction in the proof of Proposition~\ref{GSsufficient} gives a
		solution \(u\in\DD(G)\). Since \(f\) is smooth, its Fourier coefficients are
		rapidly decreasing. The estimates used in that proof show that the Fourier
		coefficients of \(u\) are also rapidly decreasing, because condition
		\emph{(DC')} causes only a polynomial loss away from the zero set, while on
		the zero set the coefficients of \(u\) are either zero or constant multiples
		of coefficients of \(f\). Hence \(u\in C^\infty(G)\), and therefore
		\[
		\mathcal A\cap C^\infty(G)\subset P(C^\infty(G)).
		\]
		This proves the equality.
	\end{proof}

	\section{The Self-Dual Structure of Vekua-Type Operators on \(\mathbb S^3\)}
	\label{Section_S3}
	
	We now analyze the case of the three-dimensional sphere
	\[
	\mathbb S^3
	:=
	\{x\in\R^4:\|x\|_2=1\}
	\cong SU(2).
	\]
	Throughout this section, \(\mathbb S^3\) is endowed with the compact Lie group
	structure inherited from \(SU(2)\), together with its normalized Haar measure.
	
	This case is not covered by the non-self-duality assumptions used in the
	preceding results. Indeed, every irreducible representation of \(SU(2)\) is
	self-dual. Thus \(\mathbb S^3\) provides the simplest non-commutative model in
	which the conjugation term cannot be treated by separating two distinct
	representation classes \([\xi]\) and \([\overline\xi]\).
	
	The representation-theoretic mechanism that appears here is not accidental. In
	the representation theory of compact connected semisimple Lie groups,
	self-duality is governed by the action of the longest Weyl group element on the
	weight lattice. In this sense, the analysis of
	\(\mathbb S^3\simeq SU(2)\) should be viewed as a model for the self-dual side
	of the theory, where the conjugation term acts inside representation blocks
	rather than between distinct blocks.
	
	The analysis below relies on the explicit Fourier structure of \(SU(2)\). In
	this setting, the conjugation symmetry relates the matrix entries
	\[
	(m,n)
	\quad\text{and}\quad
	(-m,-n)
	\]
	within each representation block. This internal symmetry allows us to derive a
	blockwise \(2\times2\) system analogous to the one obtained in the
	non-self-dual case, with the conjugate representation replaced by the symmetric
	entry in the same block.
	
	The unitary dual of \(\mathbb S^3\simeq SU(2)\) is naturally parametrized by
	\(\frac12\N_0\). For each \(\ell\in\frac12\N_0\), let
	\[
	t^\ell:\mathbb S^3\to U(2\ell+1)
	\]
	denote the corresponding irreducible unitary representation. We write its
	matrix entries as
	\[
	t^\ell_{mn},
	\qquad
	m,n\in J_\ell,
	\]
	where
	\[
	J_\ell:=\{-\ell,-\ell+1,\ldots,\ell-1,\ell\}.
	\]
	
	Let \(\{Y_1,Y_2,Y_3\}\) be the standard basis of the Lie algebra
	\(\mathfrak{su}(2)\) satisfying
	\[
	[Y_1,Y_2]=Y_3,
	\qquad
	[Y_2,Y_3]=Y_1,
	\qquad
	[Y_3,Y_1]=Y_2.
	\]

	We denote by \(D_1,D_2,D_3\) the corresponding left-invariant differential
	operators on \(\mathbb S^3\). With our convention for the matrix coefficients
	of \(SU(2)\), one has
	\[
	D_3(t^\ell_{mn})=-in\,t^\ell_{mn}.
	\]
	
	After introducing the complexified vector fields
	\[
	\partial_+:=iD_1-D_2,
	\qquad
	\partial_-:=iD_1+D_2,
	\qquad
	\partial_0:=iD_3,
	\]
	we have
	\begin{align}\label{tabela}
		\partial_+(t^\ell_{mn})
		&=
		-\sqrt{(\ell-n)(\ell+n+1)}\,t^\ell_{m,n+1},
		\nonumber\\
		\partial_-(t^\ell_{mn})
		&=
		-\sqrt{(\ell+n)(\ell-n+1)}\,t^\ell_{m,n-1},
		\nonumber\\
		\partial_0(t^\ell_{mn})
		&=
		n\,t^\ell_{mn}.
	\end{align}
	
	The self-duality of the representations is reflected in the following identity.
	For every \(\ell\in\frac12\N_0\), the conjugate representation
	\(\overline{t^\ell}\) is equivalent to \(t^\ell\), and, in the standard basis,
	one has
	\[
	\overline{t^\ell(x)_{mn}}
	=
	(-1)^{m-n}t^\ell(x)_{-m,-n}.
	\]

	Consequently, the Fourier coefficients of a distribution satisfy
	\begin{equation}\label{conj_S3}
		\widehat{\overline u}(\ell)_{mn}
		=
		(-1)^{m-n}
		\overline{\widehat u(\ell)_{-m,-n}},
	\end{equation}
	for all \(u\in\DD(\mathbb S^3)\), all
	\(\ell\in\frac12\N_0\), and all \(m,n\in J_\ell\).
	
	Thus, in the Fourier analysis of Vekua-type equations on \(\mathbb S^3\), the
	conjugation term does not connect different representation classes. Instead, it
	couples the symmetric entries
	\[
	(m,n)
	\quad\text{and}\quad
	(-m,-n)
	\]
	within the same representation block. This internal coupling replaces the pair
	\([\xi]\) and \([\overline\xi]\) used in the previous section.

	\subsection{Global Hypoellipticity} \
	\label{Subsection_S3_GH}
		
	We start the study of global hypoellipticity for Vekua-type operators on
	\(\mathbb S^3\). The first step is to derive the Fourier-side system associated
	with the equation \(Pu=f\) in the self-dual setting.
	
	Let
	\[
	L:\DD(\mathbb S^3)\to\DD(\mathbb S^3)
	\]
	be a continuous left-invariant operator preserving smooth functions. We assume
	that its symbol is diagonal in the basis described above. Thus, for each
	\(\ell\in\frac12\N_0\), we write
	\[
	\sigma_L(\ell)
	=
	\operatorname{diag}
	\big(\sigma_m(\ell)\big)_{m\in J_\ell}.
	\]
	
	Let \(p,q\in\C\), with \(p\neq0\), and consider the corresponding Vekua-type
	operator
	\[
	Pu=Lu-q\,u-p\,\overline u.
	\]
	
	Since the symbol of \(L\) is diagonal, we have
	\[
	\widehat{Lu}(\ell)_{mn}
	=
	\sigma_m(\ell)\widehat u(\ell)_{mn}.
	\]
	
	Moreover, by \eqref{conj_S3},
	\[
	\widehat{\overline u}(\ell)_{mn}
	=
	(-1)^{m-n}
	\overline{\widehat u(\ell)_{-m,-n}}.
	\]
	
	Here \(m-n\in\mathbb Z\), so the factor \((-1)^{m-n}\) is well defined.
	Therefore, the Fourier coefficient of \(Pu\) at the entry \((m,n)\) is
	\begin{equation}\label{eqn1}
		(\sigma_m(\ell)-q)\widehat u(\ell)_{mn}
		-
		p(-1)^{m-n}
		\overline{\widehat u(\ell)_{-m,-n}}
		=
		\widehat{Pu}(\ell)_{mn}.
	\end{equation}
	
	Now we write the same identity at the symmetric entry \((-m,-n)\). Since
	\[
	(-1)^{-m+n}=(-1)^{m-n},
	\]
	we get
	\[
	(\sigma_{-m}(\ell)-q)\widehat u(\ell)_{-m,-n}
	-
	p(-1)^{m-n}
	\overline{\widehat u(\ell)_{mn}}
	=
	\widehat{Pu}(\ell)_{-m,-n}.
	\]
	Taking complex conjugates gives
	\[
	\big(\overline{\sigma_{-m}(\ell)}-\overline q\big)
	\overline{\widehat u(\ell)_{-m,-n}}
	-
	\overline p\,(-1)^{m-n}
	\widehat u(\ell)_{mn}
	=
	\overline{\widehat{Pu}(\ell)_{-m,-n}}.
	\]

	Multiplying this identity by \((-1)^{m-n}\), we obtain
	\begin{equation}\label{eqn2}
		-\overline p\,\widehat u(\ell)_{mn}
		+
		\big(\overline{\sigma_{-m}(\ell)}-\overline q\big)
		\left[
		(-1)^{m-n}
		\overline{\widehat u(\ell)_{-m,-n}}
		\right]
		=
		(-1)^{m-n}
		\overline{\widehat{Pu}(\ell)_{-m,-n}}.
	\end{equation}
	
	Therefore, for each fixed \(\ell\) and each pair of symmetric entries
	\[
	(m,n)
	\quad\text{and}\quad
	(-m,-n),
	\]
	we obtain the \(2\times2\) system
	\[
	\left\{
	\begin{array}{l}
		(\sigma_m(\ell)-q)\widehat u(\ell)_{mn}
		-
		p\left[
		(-1)^{m-n}
		\overline{\widehat u(\ell)_{-m,-n}}
		\right]
		=
		\widehat{Pu}(\ell)_{mn},
		\\[0.4em]
		-\overline p\,\widehat u(\ell)_{mn}
		+
		\big(\overline{\sigma_{-m}(\ell)}-\overline q\big)
		\left[
		(-1)^{m-n}
		\overline{\widehat u(\ell)_{-m,-n}}
		\right]
		=
		(-1)^{m-n}
		\overline{\widehat{Pu}(\ell)_{-m,-n}}.
	\end{array}
	\right.
	\]
	
	The unknowns are
	\[
	\widehat u(\ell)_{mn}
	\quad\text{and}\quad
	(-1)^{m-n}\overline{\widehat u(\ell)_{-m,-n}}.
	\]
	
	When \((m,n)=(0,0)\), the symmetric entry coincides with the original one; in
	that case the same system should be understood as the real-linear system
	relating \(\widehat u(\ell)_{00}\) and its complex conjugate.
	
	The determinant of this system is
	\begin{equation}\label{Delta_S3}
		\Delta_m(\ell)
		:=
		(\sigma_m(\ell)-q)
		\big(\overline{\sigma_{-m}(\ell)}-\overline q\big)
		-
		|p|^2.
	\end{equation}
	
	We define the corresponding zero set by
	\begin{equation}\label{Z_S3}
		\mathcal Z_{\mathbb S^3}
		:=
		\left\{
		\ell\in\frac12\N_0:
		\Delta_m(\ell)=0
		\text{ for some }m\in J_\ell
		\right\}.
	\end{equation}
	
	The determinant \(\Delta_m(\ell)\) plays, in the self-dual setting of
	\(\mathbb S^3\), the same role that \(\Delta_k(\xi)\) plays in the
	non-self-dual analysis. In the study of global hypoellipticity, it is the small
	divisor that controls whether rapid decay of the Fourier coefficients of
	\(Pu\) can be transferred to the Fourier coefficients of \(u\).
	
	We first record the algebraic obstruction produced by infinitely many zeros of
	this determinant.
	
	\begin{proposition}\label{ZinfiniteSU2} 
		If the set \(\mathcal Z_{\mathbb S^3}\) is infinite, then \(P\) is not
		globally hypoelliptic. 
	\end{proposition}
	
	\begin{proof}
		Suppose that \(\mathcal Z_{\mathbb S^3}\) is infinite. Then there exist an
		increasing sequence \((\ell_j)_{j\in\N}\) in \(\frac12\N_0\) and indices
		\(m_j\in J_{\ell_j}\) such that
		\[
		\Delta_{m_j}(\ell_j)=0,
		\qquad j\in\N.
		\]
		In particular, \(\ell_j\to\infty\).
		
		We shall construct a nonsmooth distribution \(u\) satisfying \(Pu=0\). In
		order to avoid the fixed point of the involution
		\[
		(m,n)\mapsto(-m,-n),
		\]
		we choose, after discarding finitely many terms if necessary, indices
		\(n_j\in J_{\ell_j}\) such that
		\[
		(m_j,n_j)\neq(-m_j,-n_j).
		\]
		
		This is possible because \(\ell_j\to\infty\). Indeed, if \(m_j\neq0\), we may
		take \(n_j=m_j\); if \(m_j=0\), we choose any nonzero \(n_j\in J_{\ell_j}\).
		
		Set
		\[
		B_j:=\overline{\sigma_{-m_j}(\ell_j)}-\overline q,
		\qquad
		\varepsilon_j:=(-1)^{m_j-n_j}.
		\]
		Here \(m_j-n_j\in\mathbb Z\), and therefore \(\varepsilon_j=\pm1\).
		
		We define the Fourier coefficients of \(u\) only on the two symmetric entries
		\[
		(m_j,n_j)
		\quad\text{and}\quad
		(-m_j,-n_j),
		\]
		by setting
		\[
		\widehat u(\ell)_{mn}
		:=
		\begin{cases}
			B_j,
			& \text{if }\ell=\ell_j,\ m=m_j,\ n=n_j, \\[0.3em]
			\varepsilon_j p,
			& \text{if }\ell=\ell_j,\ m=-m_j,\ n=-n_j, \\[0.3em]
			0,
			& \text{otherwise}.
		\end{cases}
		\]
		
		By the polynomial growth of the symbol of \(L\), the sequence \((B_j)\) has
		at most polynomial growth. Hence the family above defines a distribution
		\(u\in\DD(\mathbb S^3)\).
		
		We claim that \(Pu=0\). For the entry \((m_j,n_j)\), equation \eqref{eqn1}
		gives
		\begin{align*}
			\widehat{Pu}(\ell_j)_{m_jn_j}
			&=
			(\sigma_{m_j}(\ell_j)-q)\widehat u(\ell_j)_{m_jn_j}
			-
			p\varepsilon_j
			\overline{\widehat u(\ell_j)_{-m_j,-n_j}}
			\\
			&=
			(\sigma_{m_j}(\ell_j)-q)B_j
			-
			p\varepsilon_j\overline{\varepsilon_j p}.
		\end{align*}
		
		Since \(\varepsilon_j=\pm1\), we have
		\[
		\varepsilon_j\overline{\varepsilon_j p}
		=
		\varepsilon_j^2\overline p
		=
		\overline p.
		\]
		
		Therefore
		\begin{align*}
			\widehat{Pu}(\ell_j)_{m_jn_j}
			&=
			(\sigma_{m_j}(\ell_j)-q)
			\big(\overline{\sigma_{-m_j}(\ell_j)}-\overline q\big)
			-
			|p|^2
			\\
			&=
			\Delta_{m_j}(\ell_j)
			=
			0.
		\end{align*}
		
		For the symmetric entry \((-m_j,-n_j)\), again using \eqref{eqn1}, we obtain
		\begin{align*}
			\widehat{Pu}(\ell_j)_{-m_j,-n_j}
			&=
			(\sigma_{-m_j}(\ell_j)-q)
			\widehat u(\ell_j)_{-m_j,-n_j}
			-
			p\varepsilon_j
			\overline{\widehat u(\ell_j)_{m_jn_j}}
			\\
			&=
			(\sigma_{-m_j}(\ell_j)-q)\varepsilon_j p
			-
			p\varepsilon_j\overline{B_j}.
		\end{align*}
		
		Since
		\[
		\overline{B_j}
		=
		\sigma_{-m_j}(\ell_j)-q,
		\]
		it follows that
		\[
		\widehat{Pu}(\ell_j)_{-m_j,-n_j}=0.
		\]
		
		For every other entry, both the corresponding coefficient of \(u\) and the
		coefficient at its symmetric entry vanish by construction. Hence all remaining
		Fourier coefficients of \(Pu\) vanish, and therefore
		\[
		Pu=0.
		\]
		
		On the other hand, \(u\) is not smooth. Indeed,
		\[
		|\widehat u(\ell_j)_{-m_j,-n_j}|
		=
		|\varepsilon_j p|
		=
		|p|\neq0
		\]
		for all \(j\), while \(\ell_j\to\infty\). Hence the Fourier coefficients of
		\(u\) are not rapidly decreasing. Therefore
		\[
		u\in\DD(\mathbb S^3)\setminus C^\infty(\mathbb S^3),
		\qquad
		Pu=0\in C^\infty(\mathbb S^3).
		\]
		
		This proves that \(P\) is not globally hypoelliptic.
	\end{proof}

	\begin{definition}\label{DC_S3}
		We say that \(P\) satisfies condition \emph{(DC\(_{\mathbb S^3}\))} if
		there exist constants \(C>0\), \(M>0\), and \(R>0\) such that
		\[
		|\Delta_m(\ell)|
		\geq
		C\jp{\ell}^{-M},
		\]
		for all \(\ell\in\frac12\N_0\) with \(\jp{\ell}\geq R\) and all
		\(m\in J_\ell\).
	\end{definition}
	
	Condition \emph{(DC\(_{\mathbb S^3}\))} immediately implies that the zero set
	\(\mathcal Z_{\mathbb S^3}\) is finite. Indeed, if \(\jp{\ell}\geq R\), then
	\[
	|\Delta_m(\ell)|\geq C\jp{\ell}^{-M}>0
	\]
	for every \(m\in J_\ell\). Thus zeros of \(\Delta_m(\ell)\) may occur only
	among the finitely many representation levels with \(\jp{\ell}<R\).
	
	\begin{theorem}\label{GH_SU2_characterization}
		Let \(L:\DD(\mathbb S^3)\to\DD(\mathbb S^3)\) be a continuous
		left-invariant operator preserving smooth functions. Assume that its symbol
		is diagonal in the standard Fourier basis of \(SU(2)\) and has polynomial
		growth. Let \(p,q\in\C\), with \(p\neq0\), and consider
		\[
		Pu=Lu-q\,u-p\,\overline u.
		\]
		Then \(P\) is globally hypoelliptic on \(\mathbb S^3\) if and only if it
		satisfies condition \emph{(DC\(_{\mathbb S^3}\))}.
	\end{theorem}
	
	\begin{proof}
		We first prove that condition \emph{(DC\(_{\mathbb S^3}\))} implies global
		hypoellipticity. Let \(u\in\DD(\mathbb S^3)\) and suppose that
		\[
		Pu=f\in C^\infty(\mathbb S^3).
		\]
		
		We shall prove that \(u\in C^\infty(\mathbb S^3)\).
		
		Fix \(\ell\in\frac12\N_0\) and \(m,n\in J_\ell\). Set
		\begin{equation}\label{Am_Bm}
			A_m(\ell):=\sigma_m(\ell)-q,
			\qquad
			B_m(\ell):=\overline{\sigma_{-m}(\ell)}-\overline q.
		\end{equation}
		Then
		\[
		\Delta_m(\ell)=A_m(\ell)B_m(\ell)-|p|^2.
		\]
		
		From \eqref{eqn1} and \eqref{eqn2}, the quantities
		\[
		\widehat u(\ell)_{mn}
		\quad\text{and}\quad
		(-1)^{m-n}\overline{\widehat u(\ell)_{-m,-n}}
		\]
		satisfy the linear system
		\[
		\begin{pmatrix}
			A_m(\ell) & -p\\
			-\overline p & B_m(\ell)
		\end{pmatrix}
		\begin{pmatrix}
			\widehat u(\ell)_{mn}\\
			(-1)^{m-n}\overline{\widehat u(\ell)_{-m,-n}}
		\end{pmatrix}
		=
		\begin{pmatrix}
			\widehat f(\ell)_{mn}\\
			(-1)^{m-n}\overline{\widehat f(\ell)_{-m,-n}}
		\end{pmatrix}.
		\]
		Hence, whenever \(\Delta_m(\ell)\neq0\), Cramer's rule gives
		\begin{equation}\label{cramer_S3_GH}
			\Delta_m(\ell)\widehat u(\ell)_{mn}
			=
			B_m(\ell)\widehat f(\ell)_{mn}
			+
			p(-1)^{m-n}\overline{\widehat f(\ell)_{-m,-n}}.
		\end{equation}
		
		Let \(N>0\) be arbitrary. Since the symbol of \(L\) has polynomial growth,
		there exist constants \(C_\sigma>0\) and \(K\in\N_0\) such that
		\[
		|\sigma_m(\ell)|\leq C_\sigma\jp{\ell}^{K},
		\qquad m\in J_\ell.
		\]
		
		Hence, for some constant \(C_B>0\),
		\[
		|B_m(\ell)|
		\leq
		|\sigma_{-m}(\ell)|+|q|
		\leq
		C_B\jp{\ell}^{K}.
		\]
		
		Since \(f\in C^\infty(\mathbb S^3)\), its Fourier coefficients are rapidly
		decreasing. Therefore, after increasing the decay order if necessary, there
		exists \(C_N>0\) such that
		\[
		|\widehat f(\ell)_{mn}|
		\leq
		C_N\jp{\ell}^{-N-M-K},
		\qquad
		|\widehat f(\ell)_{-m,-n}|
		\leq
		C_N\jp{\ell}^{-N-M},
		\]
		for all \(\ell\in\frac12\N_0\) and all \(m,n\in J_\ell\).
		
		Now assume that \(\jp{\ell}\geq R\). By condition
		\emph{(DC\(_{\mathbb S^3}\))},
		\[
		|\Delta_m(\ell)|^{-1}
		\leq
		C^{-1}\jp{\ell}^{M}.
		\]
		
		In particular, \(\Delta_m(\ell)\neq0\), and \eqref{cramer_S3_GH} applies.
		Thus
		\begin{align*}
			|\widehat u(\ell)_{mn}|
			&\leq
			|\Delta_m(\ell)|^{-1}
			\left(
			|B_m(\ell)|\,|\widehat f(\ell)_{mn}|
			+
			|p|\,|\widehat f(\ell)_{-m,-n}|
			\right)
			\\
			&\leq
			C^{-1}\jp{\ell}^{M}
			\left(
			C_B\jp{\ell}^{K} C_N\jp{\ell}^{-N-M-K}
			+
			|p| C_N\jp{\ell}^{-N-M}
			\right)
			\\
			&\leq
			C_N'\jp{\ell}^{-N}.
		\end{align*}
		
		It remains only to deal with the modes satisfying \(\jp{\ell}<R\). There are
		only finitely many such representation levels, and hence only finitely many
		Fourier coefficients are involved. Enlarging \(C_N'\), if necessary, we obtain
		\[
		|\widehat u(\ell)_{mn}|
		\leq
		C_N'\jp{\ell}^{-N},
		\]
		for all \(\ell\in\frac12\N_0\) and all \(m,n\in J_\ell\).
		
		Since \(N>0\) was arbitrary, the Fourier coefficients of \(u\) are rapidly
		decreasing. Thus \(u\in C^\infty(\mathbb S^3)\) and \(P\) is globally hypoelliptic.
		
		Conversely, suppose that condition \emph{(DC\(_{\mathbb S^3}\))} does not
		hold. We shall prove, by contraposition, that \(P\) is not globally
		hypoelliptic.
		
		If the zero set \(\mathcal Z_{\mathbb S^3}\) is infinite, then
		Proposition~\ref{ZinfiniteSU2} already implies that \(P\) is not globally
		hypoelliptic. Hence we may assume that \(\mathcal Z_{\mathbb S^3}\) is finite.
		
		Since condition \emph{(DC\(_{\mathbb S^3}\))} fails and
		\(\mathcal Z_{\mathbb S^3}\) is finite, there exist a sequence
		\(\ell_j\to\infty\) and indices \(m_j\in J_{\ell_j}\) such that
		\[
		0<
		|\Delta_{m_j}(\ell_j)|
		<
		\jp{\ell_j}^{-j},
		\qquad j\in\N.
		\]
		
		After passing to a subsequence if necessary, we choose \(n_j\in J_{\ell_j}\)
		such that
		\[
		(m_j,n_j)\neq(-m_j,-n_j)
		\]
		for every \(j\). This is possible because \(\ell_j\to\infty\).
		
		Set
		\[
		\varepsilon_j:=(-1)^{m_j-n_j},
		\]
		and define
		\[
		A_j:=A_{m_j}(\ell_j)=\sigma_{m_j}(\ell_j)-q,
		\qquad
		B_j:=B_{m_j}(\ell_j)
		=\overline{\sigma_{-m_j}(\ell_j)}-\overline q.
		\]
		
		We now define a distribution \(u\) by prescribing its Fourier coefficients on
		the symmetric pairs
		\[
		(m_j,n_j)
		\quad\text{and}\quad
		(-m_j,-n_j).
		\]
		Set
		\[
		\widehat u(\ell)_{mn}
		:=
		\begin{cases}
			B_j,
			& \text{if }\ell=\ell_j,\ m=m_j,\ n=n_j, \\[0.3em]
			\varepsilon_j p,
			& \text{if }\ell=\ell_j,\ m=-m_j,\ n=-n_j, \\[0.3em]
			0,
			& \text{otherwise}.
		\end{cases}
		\]
		
		Since the symbol of \(L\) has polynomial growth, the sequence
		\((B_j)_{j\in\N}\) has at most polynomial growth. Hence this family of Fourier
		coefficients defines a distribution
		\[
		u\in\DD(\mathbb S^3).
		\]
		
		The distribution \(u\) is not smooth. Indeed,
		\[
		|\widehat u(\ell_j)_{-m_j,-n_j}|
		=
		|\varepsilon_j p|
		=
		|p|\neq0
		\]
		for every \(j\), while \(\ell_j\to\infty\). Hence the Fourier coefficients of
		\(u\) are not rapidly decreasing.
		
		We now compute \(Pu\). At the entry \((m_j,n_j)\), using \eqref{eqn1}, we have
		\begin{align*}
			\widehat{Pu}(\ell_j)_{m_jn_j}
			&=
			A_j\,\widehat u(\ell_j)_{m_jn_j}
			-
			p\,\varepsilon_j
			\overline{\widehat u(\ell_j)_{-m_j,-n_j}}
			\\
			&=
			A_jB_j
			-
			p\,\varepsilon_j\overline{\varepsilon_jp}
			\\
			&=
			A_jB_j-|p|^2
			\\
			&=
			\Delta_{m_j}(\ell_j).
		\end{align*}
		
		At the symmetric entry \((-m_j,-n_j)\), using the second equation of the
		corresponding system, we obtain
		\[
		(-1)^{m_j-n_j}
		\overline{\widehat{Pu}(\ell_j)_{-m_j,-n_j}}
		=
		-\overline p\,B_j+B_j\overline p
		=
		0.
		\]
		
		For every other entry, both the corresponding coefficient of \(u\) and the
		coefficient at its symmetric entry vanish by construction. Thus the only
		possibly nonzero Fourier coefficients of \(Pu\) along this sequence are
		\[
		\widehat{Pu}(\ell_j)_{m_jn_j}
		=
		\Delta_{m_j}(\ell_j).
		\]
		
		These coefficients are rapidly decreasing. Indeed, for every \(N>0\), if
		\(j\geq N\), then
		\[
		|\Delta_{m_j}(\ell_j)|
		<
		\jp{\ell_j}^{-j}
		\leq
		\jp{\ell_j}^{-N}.
		\]
		
		The finitely many terms with \(j<N\) are absorbed into the constant in the
		rapid-decay estimate. Hence the Fourier coefficients of \(Pu\) decay faster
		than any negative power of \(\jp{\ell}\), and therefore \(Pu\in C^\infty(\mathbb S^3).\)
		
		Thus \(u\in\DD(\mathbb S^3)\setminus C^\infty(\mathbb S^3)\) and
		\(Pu\in C^\infty(\mathbb S^3).\) Therefore \(P\) is not globally hypoelliptic.
	\end{proof}
	
	\subsection{Global Solvability} 
	\label{Subsection_S3_GS} \

	We now turn to global solvability, keeping the notation introduced in the
	previous subsection. In particular, \(\Delta_m(\ell)\) denotes the determinant
	of the Fourier-side system \eqref{eqn1}--\eqref{eqn2}, and we write
	\[
	A_m(\ell):=\sigma_m(\ell)-q,
	\qquad
	B_m(\ell):=\overline{\sigma_{-m}(\ell)}-\overline q.
	\]
	
	When \(\Delta_m(\ell)=0\), the system cannot be inverted, and the datum must
	satisfy a compatibility condition. This leads to the following admissible space.
	
	\begin{definition}\label{A_S3}
		We denote by \(\mathcal A_{\mathbb S^3}\) the space of all distributions
		\(f\in\DD(\mathbb S^3)\) such that, for every
		\(\ell\in\frac12\N_0\), every \(m\in J_\ell\) with
		\(\Delta_m(\ell)=0\), and every \(n\in J_\ell\), one has
		\[
		B_m(\ell)\widehat f(\ell)_{mn}
		+
		p(-1)^{m-n}\overline{\widehat f(\ell)_{-m,-n}}
		=
		0.
		\]
	\end{definition}
	
	\begin{definition}
		We say that \(P\) is globally solvable on \(\mathbb S^3\) if
		\[
		P(\DD(\mathbb S^3))=\mathcal A_{\mathbb S^3}.
		\]
	\end{definition}
	
	Away from the zero set of \(\Delta_m(\ell)\), solvability requires quantitative
	control of the inverse of the Fourier-side system. This is encoded in the
	following analogue of condition \emph{(DC')}.
	
	\begin{definition}\label{DCprime_S3}
		We say that \(P\) satisfies condition \emph{(DC'\(_{\mathbb S^3}\))} if
		there exist constants \(C>0\) and \(M>0\) such that
		\[
		|\Delta_m(\ell)|
		\geq
		C\jp{\ell}^{-M},
		\]
		for all \(\ell\in\frac12\N_0\) and all \(m\in J_\ell\) such that
		\[
		\Delta_m(\ell)\neq0.
		\]
	\end{definition}
	
	\begin{theorem}\label{GS_SU2_characterization}
		Let \(L:\DD(\mathbb S^3)\to\DD(\mathbb S^3)\) be a continuous
		left-invariant operator preserving smooth functions. Assume that its symbol
		is diagonal in the standard Fourier basis of \(SU(2)\) and has polynomial
		growth. Let \(p,q\in\C\), with \(p\neq0\), and consider
		\[
		Pu=Lu-q\,u-p\,\overline u.
		\]
		Then \(P\) is globally solvable on \(\mathbb S^3\) if and only if it
		satisfies condition \emph{(DC'\(_{\mathbb S^3}\))}.
	\end{theorem}
	
	\begin{proof}
		We first prove that condition \emph{(DC'\(_{\mathbb S^3}\))} implies global
		solvability. Let \(f\in\mathcal A_{\mathbb S^3}\). We shall construct
		\(u\in\DD(\mathbb S^3)\) such that
		\[
		Pu=f.
		\]
		
		We define the Fourier coefficients of \(u\) representation by representation.
		For each fixed \(\ell\in\frac12\N_0\), we prescribe the entries of the matrix
		\[
		\widehat u(\ell)=\big(\widehat u(\ell)_{mn}\big)_{m,n\in J_\ell}.
		\]
		
		Fix \(\ell\in\frac12\N_0\) and \(m,n\in J_\ell\). If
		\[
		\Delta_m(\ell)\neq0,
		\]
		we define
		\[
		\widehat u(\ell)_{mn}
		:=
		\frac{
			B_m(\ell)\widehat f(\ell)_{mn}
			+
			p(-1)^{m-n}\overline{\widehat f(\ell)_{-m,-n}}
		}{
			\Delta_m(\ell)
		}.
		\]
		
		This is the first component of the solution of the corresponding invertible
		\(2\times2\) system.
		
		We now consider the singular case
		\[
		\Delta_m(\ell)=0.
		\]
		
		First assume that
		\[
		(m,n)\neq(0,0).
		\]
		
		Then the two entries
		\[
		(m,n)
		\quad\text{and}\quad
		(-m,-n)
		\]
		are distinct. For each orbit of the involution
		\[
		(m,n)\mapsto(-m,-n)
		\]
		contained in the zero set, we choose one representative. If \((m,n)\) is the
		chosen representative, we set
		\[
		\widehat u(\ell)_{mn}:=0,
		\]
		and define the coefficient at the symmetric entry by
		\[
		\widehat u(\ell)_{-m,-n}
		:=
		-\frac{(-1)^{m-n}}{\overline p}
		\overline{\widehat f(\ell)_{mn}}.
		\]
		
		With this choice, the equation at the entry \((m,n)\) is satisfied. Indeed,
		\begin{align*}
			A_m(\ell)\widehat u(\ell)_{mn}
			-
			p(-1)^{m-n}\overline{\widehat u(\ell)_{-m,-n}}
			&=
			-p(-1)^{m-n}
			\overline{
				-\frac{(-1)^{m-n}}{\overline p}
				\overline{\widehat f(\ell)_{mn}}
			}
			\\
			&=
			\widehat f(\ell)_{mn}.
		\end{align*}
		
		For the symmetric entry \((-m,-n)\), since
		\[
		\Delta_{-m}(\ell)=\overline{\Delta_m(\ell)}=0,
		\]
		the compatibility condition for \(f\in\mathcal A_{\mathbb S^3}\) gives
		\[
		B_m(\ell)\widehat f(\ell)_{mn}
		+
		p(-1)^{m-n}\overline{\widehat f(\ell)_{-m,-n}}
		=
		0.
		\]

		Equivalently, after taking complex conjugates,
		\[
		A_{-m}(\ell)\overline{\widehat f(\ell)_{mn}}
		+
		\overline p\,(-1)^{m-n}\widehat f(\ell)_{-m,-n}
		=
		0.
		\]
		Thus
		\[
		A_{-m}(\ell)
		\left(
		-\frac{(-1)^{m-n}}{\overline p}
		\overline{\widehat f(\ell)_{mn}}
		\right)
		=
		\widehat f(\ell)_{-m,-n}.
		\]
		Since \(\widehat u(\ell)_{mn}=0\), this is exactly the equation at the
		symmetric entry.
		
		It remains to consider the fixed point
		\[
		(m,n)=(0,0).
		\]
		This case can occur only when \(\ell\in\N_0\). If \(\Delta_0(\ell)\neq0\), it
		has already been treated by the invertible case. Assume therefore that
		\[
		\Delta_0(\ell)=0.
		\]
		The equation for the coefficient \(\widehat u(\ell)_{00}\) becomes the scalar
		real-linear equation
		\[
		A_0(\ell)z-p\overline z=\widehat f(\ell)_{00}.
		\]
		Since
		\[
		\Delta_0(\ell)
		=
		A_0(\ell)\overline{A_0(\ell)}-|p|^2
		=
		0,
		\]
		we have
		\[
		|A_0(\ell)|=|p|\neq0.
		\]
		Moreover, because \(f\in\mathcal A_{\mathbb S^3}\), the compatibility
		condition at the entry \((0,0)\) gives
		\[
		\overline{A_0(\ell)}\,\widehat f(\ell)_{00}
		+
		p\,\overline{\widehat f(\ell)_{00}}
		=
		0.
		\]
		We define
		\[
		\widehat u(\ell)_{00}
		:=
		\frac{\widehat f(\ell)_{00}}{2A_0(\ell)}.
		\]
		Then
		\begin{align*}
			A_0(\ell)\widehat u(\ell)_{00}
			-
			p\overline{\widehat u(\ell)_{00}}
			&=
			\frac{\widehat f(\ell)_{00}}{2}
			-
			\frac{p\,\overline{\widehat f(\ell)_{00}}}{2\overline{A_0(\ell)}}.
		\end{align*}
		By the compatibility condition,
		\[
		p\,\overline{\widehat f(\ell)_{00}}
		=
		-\overline{A_0(\ell)}\,\widehat f(\ell)_{00}.
		\]
		Hence
		\[
		A_0(\ell)\widehat u(\ell)_{00}
		-
		p\overline{\widehat u(\ell)_{00}}
		=
		\widehat f(\ell)_{00}.
		\]
		
		We have therefore defined Fourier coefficients satisfying the equation
		\(Pu=f\) at every matrix entry. It remains to prove that these coefficients
		define a distribution.
		
		Since \(f\in\DD(\mathbb S^3)\), there exist constants \(C_f>0\) and
		\(N_f>0\) such that
		\[
		|\widehat f(\ell)_{mn}|
		\leq
		C_f\jp{\ell}^{N_f}
		\]
		for all \(\ell\in\frac12\N_0\) and all \(m,n\in J_\ell\).
		
		In the invertible case, condition \emph{(DC'\(_{\mathbb S^3}\))} gives
		\[
		|\Delta_m(\ell)|^{-1}
		\leq
		C^{-1}\jp{\ell}^{M}
		\]
		for some constants \(C>0\) and \(M>0\). Since the symbol of \(L\) has
		polynomial growth, there exist constants \(C_B>0\) and \(K\in\N_0\) such that
		\[
		|B_m(\ell)|\leq C_B\jp{\ell}^{K}.
		\]
		Therefore
		\[
		|\widehat u(\ell)_{mn}|
		\leq
		C'\jp{\ell}^{N_f+M+K}.
		\]
		
		If \(\Delta_m(\ell)=0\) and \((m,n)\neq(0,0)\), the definition gives
		\[
		|\widehat u(\ell)_{-m,-n}|
		\leq
		\frac{1}{|p|}|\widehat f(\ell)_{mn}|
		\leq
		C'\jp{\ell}^{N_f}.
		\]
		Finally, in the fixed case \((m,n)=(0,0)\) with
		\(\Delta_0(\ell)=0\), using \(|A_0(\ell)|=|p|\neq0\), we obtain
		\[
		|\widehat u(\ell)_{00}|
		\leq
		\frac{1}{2|p|}
		|\widehat f(\ell)_{00}|
		\leq
		C'\jp{\ell}^{N_f}.
		\]
		
		Thus the family \(\widehat u(\ell)_{mn}\) has at most polynomial growth. By
		the Fourier characterization of distributions on \(\mathbb S^3\), it defines
		a distribution
		\[
		u\in\DD(\mathbb S^3).
		\]
		By construction, the Fourier coefficients of \(Pu\) and \(f\) agree for every
		\(\ell\in\frac12\N_0\) and every \(m,n\in J_\ell\). Hence
		\[
		Pu=f.
		\]
		Since \(f\in\mathcal A_{\mathbb S^3}\) was arbitrary, \(P\) is globally
		solvable.
		
		Conversely, suppose that condition \emph{(DC'\(_{\mathbb S^3}\))} does not
		hold. We shall prove, by contraposition, that \(P\) is not globally solvable.
		
		Since \emph{(DC'\(_{\mathbb S^3}\))} fails, there exist a sequence
		\(\ell_j\to\infty\) and indices \(m_j\in J_{\ell_j}\) such that
		\[
		0<
		|\Delta_{m_j}(\ell_j)|
		<
		\jp{\ell_j}^{-j},
		\qquad j\in\N.
		\]
		Indeed, if no such high-frequency sequence existed, then the nonzero values
		of \(\Delta_m(\ell)\) would satisfy a polynomial lower bound after adjusting
		the constant on finitely many low-frequency levels.
		
		Passing to a subsequence if necessary, we choose \(n_j\in J_{\ell_j}\) such
		that
		\[
		(m_j,n_j)\neq(-m_j,-n_j)
		\]
		for every \(j\). This is possible because \(\ell_j\to\infty\).
		
		Set
		\[
		\varepsilon_j:=(-1)^{m_j-n_j}.
		\]
		We define a family of Fourier coefficients by
		\[
		\widehat f(\ell)_{mn}
		:=
		\begin{cases}
			\varepsilon_j\,\overline p^{-1},
			& \text{if }\ell=\ell_j,\ m=-m_j,\ n=-n_j, \\[0.3em]
			0,
			& \text{otherwise}.
		\end{cases}
		\]
		This family is bounded, and hence it defines a distribution
		\[
		f\in\DD(\mathbb S^3).
		\]
		
		We claim that \(f\in\mathcal A_{\mathbb S^3}\). The only nonzero Fourier
		coefficients of \(f\) occur at the entries \((-m_j,-n_j)\). For these entries,
		\[
		\Delta_{-m_j}(\ell_j)
		=
		\overline{\Delta_{m_j}(\ell_j)}
		\neq0.
		\]
		Thus the nonzero coefficients of \(f\) do not occur at entries where the
		compatibility condition is imposed. Moreover, their symmetric partners are
		the entries \((m_j,n_j)\), and
		\[
		\Delta_{m_j}(\ell_j)\neq0.
		\]
		Hence the second term in the compatibility condition also cannot detect a
		nonzero coefficient of \(f\) at a zero of the determinant. Therefore no
		compatibility condition associated with a zero of \(\Delta_m(\ell)\) is
		violated, and
		\[
		f\in\mathcal A_{\mathbb S^3}.
		\]
		
		Suppose, by contradiction, that there exists \(u\in\DD(\mathbb S^3)\) such
		that
		\[
		Pu=f.
		\]
		Applying Cramer's rule to the entry \((m_j,n_j)\), where
		\(\Delta_{m_j}(\ell_j)\neq0\), gives
		\begin{align*}
			\Delta_{m_j}(\ell_j)\widehat u(\ell_j)_{m_jn_j}
			&=
			B_{m_j}(\ell_j)\widehat f(\ell_j)_{m_jn_j}
			+
			p\varepsilon_j
			\overline{\widehat f(\ell_j)_{-m_j,-n_j}}
			\\
			&=
			0+
			p\varepsilon_j\,
			\overline{\varepsilon_j\,\overline p^{-1}}
			\\
			&=
			1.
		\end{align*}
		Consequently,
		\[
		|\widehat u(\ell_j)_{m_jn_j}|
		=
		\frac{1}{|\Delta_{m_j}(\ell_j)|}
		>
		\jp{\ell_j}^{j}.
		\]
		This growth is faster than any polynomial in \(\jp{\ell_j}\), contradicting
		the Fourier characterization of distributions on \(\mathbb S^3\).
		
		Therefore no distributional solution exists for this admissible datum
		\(f\in\mathcal A_{\mathbb S^3}\). Hence \(P\) is not globally solvable. This
		proves the contraposition and completes the proof.
	\end{proof}
	
	As in the non-self-dual case, the proof of the sufficiency part also yields
	smooth solvability for smooth admissible data. If
	\(f\in\mathcal A_{\mathbb S^3}\cap C^\infty(\mathbb S^3)\), then the
	coefficients constructed above are rapidly decreasing. In the invertible case,
	condition \emph{(DC'\(_{\mathbb S^3}\))} gives only a polynomial loss. On the
	zero set of \(\Delta_m(\ell)\), the coefficients are either zero or constant
	multiples of Fourier coefficients of \(f\); in the fixed case \((0,0)\), the
	factor \(A_0(\ell)\) satisfies \(|A_0(\ell)|=|p|\neq0\). Hence the solution
	belongs to \(C^\infty(\mathbb S^3)\).
	
	\begin{corollary}\label{smooth_GS_S3}
		Under condition \emph{(DC'\(_{\mathbb S^3}\))}, the operator \(P\) is
		globally solvable in the smooth category. More precisely,
		\[
		P(C^\infty(\mathbb S^3))
		=
		\mathcal A_{\mathbb S^3}\cap C^\infty(\mathbb S^3).
		\]
	\end{corollary}
	
	\begin{proof}
		The inclusion
		\[
		P(C^\infty(\mathbb S^3))
		\subset
		\mathcal A_{\mathbb S^3}\cap C^\infty(\mathbb S^3)
		\]
		follows from the definition of the admissible space
		\(\mathcal A_{\mathbb S^3}\) and from the fact that \(P\) preserves smooth
		functions.
		
		Conversely, let
		\[
		f\in\mathcal A_{\mathbb S^3}\cap C^\infty(\mathbb S^3).
		\]
		
		By the construction in the proof of Theorem~\ref{GS_SU2_characterization},
		there exists \(u\in\DD(\mathbb S^3)\) such that \(Pu=f\). Since \(f\) is
		smooth, its Fourier coefficients are rapidly decreasing. In the invertible
		case, condition \emph{(DC'\(_{\mathbb S^3}\))} causes only a polynomial loss.
		On the zero set of \(\Delta_m(\ell)\), the constructed coefficients of \(u\)
		are either zero or constant multiples of Fourier coefficients of \(f\). In
		the fixed case \((m,n)=(0,0)\), the factor \(A_0(\ell)\) satisfies
		\[
		|A_0(\ell)|=|p|\neq0,
		\]
		and therefore no additional loss occurs. Hence the Fourier coefficients of
		\(u\) are rapidly decreasing, so
		\[
		u\in C^\infty(\mathbb S^3).
		\]
		
		Thus
		\[
		\mathcal A_{\mathbb S^3}\cap C^\infty(\mathbb S^3)
		\subset
		P(C^\infty(\mathbb S^3)).
		\]
		This proves the equality.
	\end{proof}
	
	\section{Examples and Remarks}
	\label{Section_examples}
	
	We conclude with examples illustrating the criteria obtained in the previous
	sections. The first example is posed on \(\mathbb S^3\simeq SU(2)\) and uses the
	self-dual analysis developed in Section~\ref{Section_S3}. The remaining examples
	are posed on the product group \(\mathbb S^3\times\mathbb T^1\), where the
	determinant criteria can be checked explicitly.
	
	\begin{example}
		We first exhibit a family of Vekua-type operators on \(\mathbb S^3\) for
		which the determinant \(\Delta_m(\ell)\) vanishes along infinitely many
		representation levels.
		
		Let \(r\) be a positive even integer, and let \(p,q\in\C\) satisfy
		\[
		\Re(q)=1,
		\qquad
		|q|=|p|.
		\]
		
		Choose \(a>0\) such that
		\[
		m_0:=\left(\frac{2}{a}\right)^{1/r}\in\frac12\mathbb Z.
		\]
		
		Consider
		\[
		L=a\partial_0^r,
		\qquad
		Pu=a\partial_0^r u-q\,u-p\,\overline u.
		\]
		
		Since \(r\) is even, the symbol of \(L\) satisfies
		\[
		\sigma_m(\ell)=am^r,
		\qquad
		\sigma_{-m}(\ell)=a(-m)^r=am^r.
		\]

		Therefore, using the self-dual determinant from Section~\ref{Section_S3}, we
		obtain
		\begin{align*}
			\Delta_m(\ell)
			&=
			(\sigma_m(\ell)-q)
			\big(\overline{\sigma_{-m}(\ell)}-\overline q\big)
			-|p|^2
			\\
			&=
			(am^r-q)(am^r-\overline q)-|p|^2
			\\
			&=
			a^2m^{2r}-2am^r\Re(q)+|q|^2-|p|^2
			\\
			&=
			a^2m^{2r}-2am^r.
		\end{align*}

		For \(m=m_0\), we have \(am_0^r=2\), and hence
		\[
		\Delta_{m_0}(\ell)=0.
		\]
		
		Since \(m_0\in J_\ell\) for infinitely many
		\(\ell\in\frac12\N_0\), the zero set \(\mathcal Z_{\mathbb S^3}\) is
		infinite. By Proposition~\ref{ZinfiniteSU2}, \(P\) is not globally
		hypoelliptic.
		
		Notice also that the unperturbed operator \(L\) is not globally
		hypoelliptic. Indeed, since \(0\in J_\ell\) for every
		\(\ell\in\N_0\), one has
		\[
		\sigma_0(\ell)=0
		\]
		for infinitely many \(\ell\). 
		
		Thus the symbol of \(L\) has infinitely many
		zeros at high frequencies, which violates the usual lower-bound criterion for
		global hypoellipticity of diagonal operators; see, for instance,
		Theorem~3.2 of \cite{da2025diagonal}.
	\end{example}
	
	\begin{example}
		Let
		\[
		G=\mathbb S^3\times\mathbb T^1.
		\]

		Let \(a,q\in\R\setminus\{0\}\), let \(p\in\C\setminus\{0\}\), and let
		\(r\in\N\) be odd. Assume that
		\[
		q\notin\frac12\mathbb Z
		\qquad\text{and}\qquad
		q\pm |p|\notin\frac12\mathbb Z.
		\]

		Consider
		\[
		L=\partial_0+iaD_t^r,
		\qquad
		Pu=\partial_0u+iaD_t^r u-q\,u-p\,\overline u.
		\]
		
		The unitary dual of \(G\) is parametrized by
		\[
		\widehat G\simeq \frac12\N_0\times\mathbb Z.
		\]
		
		For \((\ell,k)\in\widehat G\) and \(m,n\in J_\ell\), the symbol of \(L\) is
		\[
		\sigma_m(\ell,k)=m+iak^r.
		\]
		
		The conjugate representation of the class \((\ell,k)\) is \((\ell,-k)\).
		Since \(r\) is odd, we have
		\[
		\sigma_m(\ell,-k)
		=
		m-iak^r
		=
		\overline{\sigma_m(\ell,k)}.
		\]
		
		Thus \(L\) is diagonal in the sense of Definition~\ref{diagonal}.
		
		For this operator, the determinant \eqref{discriminant} is
		\begin{align*}
			\Delta_m(\ell,k)
			&=
			(m+iak^r-q)(m+iak^r-\overline q)-|p|^2
			\\
			&=
			(m+iak^r-q)^2-|p|^2
			\\
			&=
			m^2-a^2k^{2r}-2qm+q^2-|p|^2
			+
			2iak^r(m-q),
		\end{align*}
		where we used that \(q\in\R\).
		
		We first consider \(k\neq0\). From the imaginary part of the last expression,
		we obtain
		\[
		|\Delta_m(\ell,k)|
		\geq
		\left|2ak^r(m-q)\right|
		\geq
		2|a|\,|m-q|.
		\]
		
		Since \(q\notin\frac12\mathbb Z\), the distance from \(q\) to
		\(\frac12\mathbb Z\) is positive. Hence there exists \(c_1>0\) such that
		\[
		|m-q|\geq c_1
		\]
		for every \(m\in\frac12\mathbb Z\). Therefore,
		\[
		|\Delta_m(\ell,k)|
		\geq
		2|a|c_1,
		\qquad k\neq0.
		\]
		
		Now let \(k=0\). Then
		\[
		\Delta_m(\ell,0)
		=
		(m-q)^2-|p|^2.
		\]
		
		Thus
		\[
		\Delta_m(\ell,0)=0
		\quad\Longleftrightarrow\quad
		m=q\pm |p|.
		\]
		
		By assumption,
		\[
		q\pm |p|\notin\frac12\mathbb Z,
		\]
		so \(\Delta_m(\ell,0)\neq0\) for every \(m\in\frac12\mathbb Z\). Moreover,
		since
		\[
		|(m-q)^2-|p|^2|\to\infty
		\qquad\text{as } |m|\to\infty,
		\]
		there exists \(C_2>0\) such that
		\[
		|\Delta_m(\ell,0)|\geq C_2
		\]
		for every \(m\in\frac12\mathbb Z\).
		
		Setting
		\[
		C:=\min\{2|a|c_1,C_2\}>0,
		\]
		we obtain
		\[
		|\Delta_m(\ell,k)|\geq C
		\]
		for all \((\ell,k)\in\widehat G\) and all \(m\in J_\ell\). 
		
		Hence \(P\)
		satisfies condition \emph{(DC)}. By Proposition~\ref{firstprop}, the
		Vekua-type operator \(P\) is globally hypoelliptic.
		
		In contrast, the unperturbed operator \(L\) is not globally hypoelliptic.
		Indeed, for \(k=0\) and \(m=0\), which occurs for every
		\(\ell\in\N_0\), one has
		\[
		\sigma_0(\ell,0)=0.
		\]
		Thus the symbol of \(L\) has infinitely many zeros at high frequencies, which
		violates the usual lower-bound criterion for global hypoellipticity of
		diagonal operators; see, for instance, Theorem~3.2 of
		\cite{da2025diagonal}. This example shows that a Vekua-type perturbation can
		restore global hypoellipticity.
	\end{example}
	
	\begin{example}
		We now give an example of a Vekua-type operator that is globally solvable
		but not globally hypoelliptic.
		
		Let
		\[
		G=\mathbb S^3\times\mathbb T^1.
		\]
		Let \(a\in\R\setminus\{0\}\), let \(q\in i\R\setminus\{0\}\), and let
		\(p\in\C\setminus\{0\}\) satisfy
		\[
		|p|=|q|.
		\]
		Consider
		\[
		L=\partial_0+iaD_t,
		\qquad
		Pu=\partial_0u+iaD_tu-q\,u-p\,\overline u.
		\]
		
		The unitary dual of \(G\) is parametrized by
		\[
		\widehat G\simeq \frac12\N_0\times\mathbb Z.
		\]
		For \((\ell,k)\in\widehat G\) and \(m,n\in J_\ell\), the symbol of \(L\) is
		\[
		\sigma_m(\ell,k)=m+iak.
		\]
		Since
		\[
		\sigma_m(\ell,-k)=m-iak=\overline{\sigma_m(\ell,k)},
		\]
		the operator \(L\) is diagonal in the sense of Definition~\ref{diagonal}.
		
		Writing \(q=ib\), with \(b\in\R\setminus\{0\}\), and using \(|p|=|q|=|b|\),
		we obtain
		\begin{align*}
			\Delta_m(\ell,k)
			&=
			(m+iak-ib)(m+iak+ib)-|p|^2
			\\
			&=
			(m+iak)^2+b^2-|p|^2
			\\
			&=
			(m+iak)^2.
		\end{align*}
		In particular,
		\[
		\Delta_0(\ell,0)=0
		\]
		for every \(\ell\in\N_0\). These zeros occur along the self-dual classes
		\((\ell,0)\). Therefore Proposition~\ref{Zinfinite} does not apply directly,
		since that result concerns infinitely many zeros along non-self-dual classes.
		Instead, repeating the construction used in Proposition~\ref{ZinfiniteSU2} on
		the zero toroidal frequencies, one obtains a distribution
		\(u\in\DD(G)\setminus C^\infty(G)\) such that
		\[
		Pu=0.
		\]
		Hence \(P\) is not globally hypoelliptic.
		
		We now verify the corresponding condition \emph{(DC')}. Since
		\[
		\Delta_m(\ell,k)=(m+iak)^2,
		\]
		we have
		\[
		|\Delta_m(\ell,k)|
		=
		|m+iak|^2
		=
		m^2+a^2k^2.
		\]
		Moreover,
		\[
		\Delta_m(\ell,k)=0
		\quad\Longleftrightarrow\quad
		m=0
		\text{ and }
		k=0.
		\]
		Thus, if \(\Delta_m(\ell,k)\neq0\), then either \(k\neq0\), or
		\(k=0\) and \(m\neq0\).
		
		If \(k\neq0\), then
		\[
		|\Delta_m(\ell,k)|
		=
		m^2+a^2k^2
		\geq
		a^2.
		\]
		If \(k=0\) and \(m\neq0\), then \(m\in\frac12\mathbb Z\setminus\{0\}\), and
		therefore
		\[
		|\Delta_m(\ell,0)|
		=
		m^2
		\geq
		\frac14.
		\]
		Consequently, setting
		\[
		C:=\min\left\{a^2,\frac14\right\}>0,
		\]
		we obtain
		\[
		|\Delta_m(\ell,k)|\geq C
		\]
		for every \((\ell,k)\in\widehat G\) and every \(m\in J_\ell\) such that
		\(\Delta_m(\ell,k)\neq0\).
		
		Thus condition \emph{(DC')} holds. In this explicit product setting, the
		Fourier-side construction used in the solvability proofs above applies
		directly: away from the zero set one solves by dividing by
		\(\Delta_m(\ell,k)\), while on the zero set one imposes the corresponding
		algebraic compatibility condition. Hence \(P\) is globally solvable on the
		natural admissible space.
		
		Therefore \(P\) is globally solvable but not globally hypoelliptic.
	\end{example}
	
	\begin{remark}
		The analysis of the self-dual case depends on an explicit description of how
		conjugation acts inside each representation block. For
		\(\mathbb S^3\simeq SU(2)\), this role is played by the identity
		\[
		\overline{t^\ell_{mn}}
		=
		(-1)^{m-n}t^\ell_{-m,-n}.
		\]
		For other compact Lie groups with self-dual irreducible representations, an
		analogous blockwise description would be needed in order to extend the
		self-dual part of the present analysis.
	\end{remark}

	\section*{Acknowledgements}
	
	The authors gratefully acknowledge the support of the National Council for
	Scientific and Technological Development -- CNPq, Brazil, through the 
	grants: 316850/2021-7, 402669/2024-0, and 420814/2025-6.
	The second author is also grateful to Prof. Michael Ruzhansky and to the Ghent 
	Analysis and PDE Center  for their hospitality during his visit to Ghent University 
	from December 2025  to February 2026. The discussions held during this visit 
	contributed to the development of the ideas leading to the present work.

	\bibliographystyle{amsplain}
	\bibliography{references}
	
\end{document}